\documentclass[11pt,a4paper]{article}

\usepackage[a4paper,hcentering,vcentering]{geometry}
\usepackage[utf8]{inputenc}
\usepackage{nameref,hyperref}
\usepackage[right]{lineno}
\usepackage{graphicx}
\usepackage{tikz-cd}
\usepackage{url}

\usepackage{amsthm,amsmath,amsfonts,mathtools}
\newtheorem{Pro}{Proposition}[section]
\newtheorem{Hyp}{Hypothesis}
\newtheorem{Lem}{Lemma}[section]

\newtheorem{Thm}{Theorem}[section]
\newtheorem{Def}{Definition}[section]
\theoremstyle{definition}

\newtheorem{Exp}{Exemple}

\theoremstyle{remark}
\newtheorem{Rmk}{Remark}

\newcommand{\C}{\ensuremath{\mathbb{C}}}

\newcommand{\R}{\ensuremath{\mathbb{R}}}
\newcommand{\Sbb}{\ensuremath{\mathbb{S}}}

\newcommand{\zb}{\ensuremath{\mathbf{z}}}
\newcommand{\Zb}{\ensuremath{\mathbf{Z}}}

\DeclareMathOperator{\RS}{\hat{\mathbb{C}}}
\DeclareMathOperator{\CP}{\mathbb{CP}^{k}}
\DeclareMathOperator{\HR2N}{H_{\mathbb{R}^{2n}}}

\DeclareMathOperator{\HCPN2}{H_{\mathbb{CP}^{n-2}}}

\DeclareMathOperator{\vz}{\mathbf{z}}
\newcommand{\norm}[1]{\lVert {#1} \rVert}

\begin{document}

\title{Choreographic Holomorphic Spheres with Application\\
    to Hamiltonian Systems of $N$-Vortex Type} 

\author{Qun WANG\\
CEREMADE, Universit{\'e} Paris Dauphine, Paris Sciences et Lettres\\
\url{wangqun@ceremade.dauphine.fr}}

\date{}

\maketitle

\subsection*{Abstract}
We study the existence of (relative) simple choreographies for a class of Hamiltonian systems describing the interaction of particles in the plane motivated mainly by the n-vortex type problem. In particular, by constructing choreographic pseudo-holomorphic spheres, we prove that there exist infinitely many non-trivial relative choreographies for the identical n-vortex problem arising from both the Euler equation and the Gross-Pitaevskii equation.


\tableofcontents 

\section{Introduction}

\subsection{Interactive System of N-Vortex Type}

\paragraph{}Let us consider a class of Hamiltonian systems in $\R^{2n}$ of the form

\begin{align*}
\dot\Zb (t) = \mathcal{J}_{\R^{2N}}\nabla \HR2N(\Zb(t)),
\quad \Zb = (z_1,z_2,...,z_n),\quad  z_i=(x_i,y_i)\in \R^2
\end{align*}
with
\begin{align}
\label{System: H1}
\HR2N(\mathbf{Z})= \sum_{i=1}^{n}\alpha_i V(|z_i|^2) + \sum_{1 \leq i< j\leq n} \beta_{ij}F(|z_i -z_j|^2). \tag{H1}
\end{align}
Here 
\begin{itemize}
\item$z_i = (x_i, y_i)$ is the position of the $i^{th}$ particle in the plane;
\item $\nabla \HR2N$ is the gradient of $\HR2N$;
\item $\mathcal{J}_{\mathbb{R}^{2n}}$ is the standard complex structure;   
\begin{align*}
\mathcal{J}_{\mathbb{R}^{2n}} = 
\begin{bmatrix}
\mathbb{J} &       &      \\  
          & \ddots &      \\
          &       & \mathbb{J}
\end{bmatrix}, \quad 
\mathbb{J} =
\begin{bmatrix}
0  & 1 \\
-1 & 0
\end{bmatrix}
\end{align*}
\item $F$ is smooth in $\mathbb{R}^{2n} \setminus \Delta$, where 
\begin{align*}
\Delta =  \bigcup_{1\leq i <j\leq n}\{\Zb \in \mathbb{R}^{2n}| z_i = z_j\}
\end{align*}
\end{itemize}
Such a system describes the motion of $n$ particles in the plane, driven by a radial potential $V$ and an interaction function $F$ depending only on the mutual distance of each pair of particles. Here $\alpha_i$ and $\beta_{ij}$ are parameters (mass, vorticity, charge...), which might vary with indices $i$ and $j$. The velocity vector field of each particle splits into two parts. The first part depends on the absolute position of the particle, and generates a uniform rotation. The second part depends on relative positions of the particles whose behaviour is thus more complicated. 

\paragraph{}In below we discuss some well known models in hydrodynamics and quantum mechanics as examples of such Hamiltonian systems. We will be primarily interested in vortex-like systems, but part of the coming study holds in larger generality. 
\begin{Exp}[The identical n-vortex problem from Euler equation]
\label{Exp: HYD}
Let 
\begin{align*}
V=0, \beta_{ij}=-\frac{1}{4\pi}, 1\leq i< j\leq n, F(\eta) = \log|\eta|^2 
\end{align*}
In this case \eqref{System: H1} becomes
\begin{align*}
\label{System: HYD}
\HR2N(\mathbf{Z}) =- \frac{1}{4\pi}\sum_{1\leq i<j \leq n}\log|z_i-z_j|^2
\tag{n-vortex Euler}
\end{align*}
The motion of ideal flow is governed by the Euler equation
\begin{align*}
\mathbf{u}_t+ \mathbf{u} \cdot \nabla\mathbf{u} = - \nabla p
\end{align*}
This equation could be transformed into a vorticity equation \cite{marchioro2012mathematical} 
\begin{align*}
\frac{D\mathbf{\omega}}{Dt} = \mathbf{\omega}\cdot \nabla \mathbf{u}
\end{align*}
where
\begin{align*}
\mathbf{\omega} = \nabla \times \mathbf{u}=(\partial_{y}z-\partial_{z}y, \partial_{z}x-\partial_{x}z,\partial_{x}y-\partial_{y}x)
\end{align*}
Helmholtz \cite{helmholtz} considered the perpendicular section of infinitely thin, straight, parallel vortex filaments with identical constant vorticity with a plane, which turns out to be the n-vortex model. Kirchhoff \cite{kirchhoff1876vorlesungen} has found the Hamiltonian structure for such a system, with the above Hamiltonian.
\end{Exp}
\begin{Exp}[The identical n-vortex problem from Gross-Pitaevskii equation]
Let 
\begin{align*}
\alpha_i = -\frac{1}{2}\mu , 1\leq i \leq n, V(\eta) =\log\frac{1}{1-|\eta|^2}, \beta_{ij}=-\frac{1}{2}\lambda , 1\leq i< j\leq n, F(\eta) = \log|\eta|^2 
\end{align*}
\label{Exp: BEC}
\begin{align*}
\label{System: BEC}
\HR2N(\mathbf{Z}) =-\frac{1}{2}(\mu\sum_{i=1}^{n} \log\frac{1}{1-|z_i|^2} + \lambda \sum_{i<j}\log|z_i-z_j|^2)
\tag{n-vortex BEC}
\end{align*}
This Hamiltonian system describes the motion of vortices in Bose Einstein condensation (BEC). It can be observed by experiments, either via a harmonical trap  \cite{fetter2009rotating} or via a hard wall container \cite{aftalion2002shape}. This system is a 2D reduction of the Gross-Pitaevskii partial differential equation concerning the ground state of a quantum system of identical bosons. 
Here the topological charge of each vortex is fixed to be $1$, $\mu>0$ is the precession of trap center, and $\lambda>0$ is the interaction strength. The case $\mu=0$ corresponds to the classical identical n-vortex problem in hydrodynamics given in example \ref{Exp: HYD}.
\end{Exp}
\begin{Exp}[The n-site problem from nonlinear Schr{\"o}dinger equation]
\label{Exp: NLS}
Let 
\begin{align*}
\alpha_i = -\frac{1}{4}, 1\leq i \leq n, V(\eta) = |\eta|^4, \beta_{ij}=-\frac{1}{2}\delta_{ij}, 1\leq i< j\leq n, F(\eta) = \log|\eta|^2 
\end{align*}
where 
\begin{align*}
\delta_{ij} = 
\begin{cases}
0 &\text{if } i - j > 1 \mod n;   \\
1 &\text{if } i- j = 1 \mod n.  
\end{cases}
\end{align*}
In this case \eqref{System: H1} becomes
\begin{align*}
\label{System: NLS}
\HR2N(\mathbf{Z}) = \frac{1}{2} \sum_{j=1}^{n} ( \frac{1}{2}|z_j|^{4}-|z_{j+1}-z_j|^2)
\tag{n-Sites NLS}
\end{align*}
with the convention that $z_{n+1}= z_1$. This Hamiltonian system describes a simplified model for a lattice of coupled harmonic oscillators. Here $z_j = z_j(t)\in \mathbb{R}^{2} \equiv \mathbb{C}$ is the complex mode amplitude of the oscillator at site $j$. This system can be seen as a standard finite difference approximation to the continuous cubic Schr{\"o}dinger equation:
\begin{align*}
i \mathbf{Z}_t + |\mathbf{Z}|^2 \mathbf{Z} + \mathbf{Z}_{xx} = 0 
\end{align*}
For more details, see \cite{chris2003discrete}.
\end{Exp}
\paragraph{}System \eqref{System: H1} is invariant under rotation, thus the entity 
\begin{align}
\label{Def: I}
I(\mathbf{Z}(t)) = \sum_{1\leq i\leq n}|z_i(t)|^2
\end{align}
which generates rotations, is a first integral; it is called the \emph{moment of inertia} and is an analogue of the angular momentum of particle mechanics (remember that in our system $x_i$ and $y_i$ are pairs of conjugate coordinates, while in particle mechanics positions are conjugate to momenta). Moreover, the Hamiltonian is autonomous, thus $\HR2N$ itself is another first integral. As a result, for $n=2$ the system is Liouville integrable. However, for $n > 2$, this Hamiltonian system is in general \textbf{not} integrable and chaotic behaviours arise, at least conjecturally. 
\paragraph{}An exception is the case when $V$ is a constant (For instance in example \ref{Exp: HYD}, one has $V\equiv 0$). In the absence of dynamics due to $V$, the system is in addition invariant under translation. As a result,  
\begin{align}
P(\mathbf{Z}(t)) = \sum_{1\leq i\leq n}\Re(z_i(t)) =\frac{1}{n}\sum_{1\leq i\leq n}x_i(t), \quad Q(\mathbf{Z}(t)) = \sum_{1\leq i\leq n}\Im(z_i(t)) =\frac{1}{n}\sum_{1\leq i\leq n} y_i(t)
\end{align}
are also first integrals. $H,I, P^2+Q^2$ are three independent first integrals in involution. Taking the n-vortex problem in hydrodynamics for example, for $n=3$ the system is integrable \cite{poincare1893theorie} while non-integrable when $n>3$ in general \cite{ziglin1980nonintegrability, koiller1989non, castilla1993four}. 

\paragraph{}
In any case one can take advantage of the symmetry of rotation of \eqref{System: H1} by carrying out the symplectic reduction. Since $I=0$ is trivial, we may assume for example that the parameter $I=1$. The reduction is summarised by this diagram:

\begin{tikzcd}
  &&&& \Sbb^1 \arrow[hookrightarrow]{d} \\
  &&&\mathbb{R}^{2n}  \arrow[hookleftarrow]{r}{I= 1} &\Sbb^{2n-1}
  \arrow [twoheadrightarrow]{d} \\
  &&&             & \CP
\end{tikzcd}
\\
where $\Sbb^{2n-1}$ is the unit sphere of $\R^{2n}$, and the quotient of $\Sbb^{2n-1}$ by rotations identifies to the complex projective space $\mathbb{CP}^{n-1}$ of complex lines in $\R^{2n} \equiv \C^n$. As is well known, the standard Hermitian product on $\C^n$ induces a Kähler metric on $\mathbb{CP}^{n-1}$, whose imaginary part defines the opposite of the standard symplectic structure of $\mathbb{CP}^{n-1}$. \\

If furthermore $V=cst$, then one can reduce one more degree of freedom by fixing the center $(P,Q)$ at $(0,0)$ and the reduced phase space becomes $\mathbb{CP}^{n-2}$, as is summarised by the diagram:

\begin{tikzcd}
&&&&                     & \mathbb{S}^1 \arrow[hookrightarrow]{d} \\
&&& \mathbb{R}^{2n}  \arrow[hookleftarrow]{r}{P=Q=0}&\mathbb{R}^{2n-2}  \arrow[hookleftarrow]{r}{I= 1} &  \mathbb{S}^{2n-3}                \arrow [twoheadrightarrow]{d}      \\
&&&&             & \mathbb{CP}^{n-2}                     \\              
\end{tikzcd}                            
\\
In both cases, the Hamiltonian descends to a reduced Hamiltonian $H$ on the quotient, thus induces a reduced Hamiltonian system
\begin{align}
\label{System: H2}
&\dot{\vz}(t) = \mathbf{X}_{H}(\vz(t)), \quad H: \mathbb{CP}^{n-1}\rightarrow \mathbb{R},\quad \vz\in \mathbb{CP}^{n-1}, \text{ when $V\neq cst$ }\notag\\
&\dot{\vz}(t) = \mathbf{X}_{H}(\vz(t)), \quad H: \mathbb{CP}^{n-2}\rightarrow \mathbb{R},\quad \vz\in \mathbb{CP}^{n-2},\text{ when $V= cst$}
\tag{H2}
\end{align}
It is these vector fields that will be the primary source of interest in our study, aimed at finding symmetric periodic orbits.

\subsection{Absolute and Relative Periodic Orbits}
\paragraph{}
In his foundational work on the $3$-body problem, Poincaré famously showed the importance of periodic solutions, which provides important information for understanding such a complicated dynamics. One possible way of finding periodic orbits is to study orbits bifurcating from known orbits. For example methods based on the Lyapunov centre theorem around relative equilibria, or based on superposition principles (i.e. substituting a number of bodies in small relative equilibrium configuration for another body) have been applied successfully to find non-equilibrium periodic solutions in the $n$-vortex problem from Euler equation \cite{borisov2004absolute,carvalho2014lyapunov,bartsch2016periodic,bartsch2017global}. Even some particular qualitative properties (for instance some discrete symmetry of these orbits) could be deduced for these solutions. Meanwhile, these perturbative methods usually provide new solutions only locally, in a neighborhood of the known solution.

Weinstein has conjectured that many conservative systems have periodic orbits \cite{hofer1992weinstein}. For the three-body problem, Poincaré even conjectured that periodic orbits are dense among bounded motions. We may well believe in the analogous conjecture for systems \eqref{System: H1}, but this currently seems totally out of reach. The case of the identical $n$-vortex problem (example~\ref{Exp: HYD} even non-identical vortices but sharing the same sign) is a particularly favorable setting, since all orbits are bounded in $\R^{2n}$ and collision-free (i.e. bounded away from the generalised diagonal $\Delta = \cup_{i\neq j} \{z_i=z_j\})$): indeed, the boundedness comes from the conservation of the moment of inertia~\eqref{Def: I}, while the absence of collisions comes from the conservation of the Hamiltonian, all of whose terms are upper bounded along a given integral curve, and thus lower bounded.

\begin{Rmk}
    Note that in our case (as for all rotation-invariant mechanical systems) there are two competing notions of periodic orbits: one for the initial system~\eqref{System: H1} in $\R^{2n}$, and one for the reduced system~\eqref{System: H2} in $\mathbb{CP}^{n-1}$. Corresponding periodic orbits are called \emph{absolute} and \emph{reduced}, respectively. The projection on $\mathbb{CP}^{n-1}$ of an absolute periodic orbit yields a reduced one. Conversely, a lift $\Zb(t)$ to $\R^{2n}$ of a reduced $T$-periodic orbit $\zb(t)$ is usually not periodic; it is called \emph{relative}. Rather, there exists an angle $\alpha^o \in \R$ (uniquely defined modulo $2\pi$) such that $\Zb(T)$ is obtained from $\zb(T)$ by the rigid rotation (acting diagonally on all vortices) of angle $\alpha^o$. A way to recover periodicity for the lifted orbit is to look at it in a rotating frame of reference. If $\alpha(t)$ is the rotation (continuous with respect to time) of the new frame, this amounts to introducing the path $\tilde\zb(t) = e^{-i\alpha(t)}\zb(t)$. Assuming that $\alpha(0) = 0$, the new path $\tilde\zb$ is periodic provided $\alpha(T) \in -\alpha^o + 2 \pi \mathbb{Q}$. The frame may rotate non-uniformly with respect to time, or even non-monotonically. For the sake of simplicity, one usually considers only frames which rotate uniformly. But even then there are countably many such frames (determined by its angle of rotation during the period of the reduced orbit). Similar argument holds for reduced system~$\mathbb{CP}^{n-2}$, except that the center of a lifted orbit might be other than the origin.
\end{Rmk}

For example in \cite{wang2018relative} periodic orbits for the $n$-vortex problem in the plane have been found. Unfortunately, it is a difficult problem to distinguish periodic orbits on a given energy level, when these are determined through implicit methods, instead of explicit constructions. Let us consider the following simple example of RPO for the BEC identical 4-vortex problem (\textbf{figure} \ref{fig:twoconfig}). In the left $(1234)$ configuration, the distances of the four vortices are, roughly speaking, of the same scale. As a result, the motion will be that the four vortices confine themselves in a relatively small cluster and chase each other therein, while the cluster as an entity rotates together around the origin $\textbf{O}$; However in the right $(123)(4)$ configuration, the $4^{th}$ vortex is relatively far away from the others, hence the behaviour will be that the three vortices form their own cluster, thus this cluster and the $4^{th}$ vortex rotate as two clusters around the origin $\textbf{O}$. Note that we can adjust the distances to make them of same energy level $H$ and of same angular momentum $I$. So, a constant issue we have is the triviality issue: one has to show  that  the periodic orbits we find, absolute or relative, are distinct from well known ones (and in particular that their reduction is not a fixed point). More precisely, we could ask the following questions:
\begin{itemize}
\item Can one find an orbit that looks like $(1234)$, instead of $(123)(4)$?
\item Further more, suppose an orbit looks like $(1234)$ has been found, can one distinguish this orbit from relative equilibria, i.e., a square configuration rotating around its center of vorticity in certain rotating frame?
\end{itemize}

\medskip The above example explains our motivation in this work: since the dynamical systems we study are in general non-integrable (this non-integrability is rarely trivial and often requires special arguments, which we will not develop here), it is hopeless to characterise orbits by quadratures and eliminations. Nevertheless, we claim that it is possible to find non-trivial periodic orbits of \eqref{System: H1} with some abstract, variational methods, and even some more specific classes of orbits, displaying a rich discrete symmetry group. 

Again, the study of n-body problem in the plane sheds some light on our problem. In \cite{poincare1896solutions}, Poincar{\'e} had understood the difficulty of minimisation for the Lagrangian action functional in a given homotopy class, due to the possibility of collisions. Since then there have been at least two perspectives to add topological constraints. These constraints serve not only as the guarantee of coercivity, but also as ways to distinguish different orbits such like the ones we see in \textbf{figure} \ref{fig:twoconfig}. More precisely, we may consider:
\begin{itemize}
\item{\textbf{Homotopic Constraint}}: it is requested that the orbit fall in a special free homotopic class \cite{gordon1977minimizing,montgomery1998n, venturelli2001caracterisation};
\item{\textbf{Symmetry Constraint}}: it is requested that the orbit be invariant under action of a special symmetry group \cite{degiovanni1987periodic,zelati1990periodic,
chenciner2000remarkable,chenciner2003action,chenciner2009unchained}.
\end{itemize}
\begin{figure}
\begin{center}
\includegraphics[width=80mm,scale=0.5]{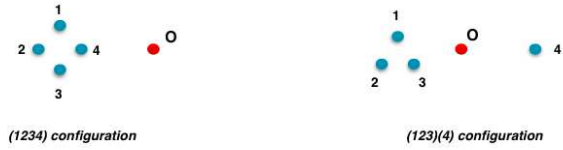}
\caption{Two configurations of same $H$ and $I$}
\label{fig:twoconfig}
\end{center}
\end{figure}
If the orbits found meet these constraints, then we will have gained qualitative understanding of them. In this article, for the sake of simplicity we will focus on a special symmetry constraint, namely the \textit{simple choreographic symmetry}, and study the existence of relative periodic orbits with such symmetry.

\section{Choreographic Holomorphic Spheres}
\subsection{Absolute and Relative Choreographies}

\subsubsection{Choreography}
\paragraph{}
We are interested in relative periodic solutions of the system \eqref{System: H1} that satisfy some symmetry condition, namely the choreographic symmetry. The study of choreographies begins with the seminal paper of Chenciner and Montgomery \cite{chenciner2000remarkable} on the proof of existence of the figure-eight solution for the 3-body problem, following the earlier numerical experiment of \cite{moore1993braids}. We denote the set of $2\pi$-periodic continuous loops by 
\begin{align*}
\Lambda = \{\mathbf{Z}\in \mathcal{C}(\mathbb{S}^1, \mathbb{R}^{2n})| \mathbf{Z}(0) = \mathbf{Z}(2\pi)\}, \quad \mathbb{S}^1 = \mathbb{R}/2\pi\mathbb{Z}.
\end{align*}
Define the time translation and the circular permutation
\begin{align}
&\tau: \mathbb{S}^1 \rightarrow\mathbb{S}^1\quad \quad\quad
 \tau(t) = \frac{2\pi}{n} +t\\
&\tilde{\sigma}: \mathbb{R}^{2n} \rightarrow \mathbb{R}^{2n} \quad  (z_1,z_2,...,z_{n-1}, z_n) \xrightarrow{\tilde{\sigma}} (z_n,z_1,...,z_{n-2},z_{n-1})
\end{align}
and 
\begin{align*}
g: \Lambda \rightarrow \Lambda\quad \quad (g \mathbf{Z})(t)= \tilde{\sigma} \mathbf{Z}(\tau^{-1}t)
\end{align*}
We are interested in the fixed points of $g$, namely free loops satisfying
\begin{align} 
z_{i+1}(t+\frac{T}{n}) = z_i(t)
\end{align} 
\begin{Def}
We call a loop $\Zb \in \Lambda$
\begin{itemize}
\item a \textbf{choreography}, if $g \Zb = \Zb$;
\item a \textbf{centred choreography}, if $\mathbf{Z}(t)$ is a choreography and 
\begin{align}
P(\mathbf{Z}(t)) = Q(\mathbf{Z}(t)) = 0, \forall t\in [0,2\pi]
\end{align}
\end{itemize}
\end{Def}

\paragraph{}
This choreographic symmetry means that particles describe the same orbit in the plane, and are merely separated by a fixed amount of time. One may define more complicated kinds of choreographies, corresponding to permutations $\tilde\sigma$ splitting into several cycles, but we will not consider such so-called multiple choreographies. We will thus have omit the adjective ``simple" in this article. 
\paragraph{}
The simplest choreography is the regular $n$-gon relative equilibrium, namely the motion along which the $n$ particles sit on the $n$ vertices of a regular $n$-gon, and rotate uniformly. A direct elementary computation shows that such solutions exist in the identical $n$-vortex problem (``Thomson configuration") or in the identical $n$-body problem (the bodies should additionally then be given the right velocities, without which the motion is homographic). The Trojan satellites in the Solar system are close to an equilateral configuration with Jupiter and the Sun.

\subsubsection{Reduced Choreography in $\mathbb{CP}^{n-1}$}
\paragraph{}
Similarly to when we weakened the notion of periodic orbit by introducing the idea of reduced or relative periodic orbits, it is natural to consider solutions which are choreographic for the reduced dynamics, in the sense which follows, and which primarily uses the existence of an action of the symmetric group on $\mathbb{CP}^{n-1}$. Denote the set of $2\pi$-parameterised continuous loops in $\mathbb{CP}^{n-1}$ by 
\begin{align*}
\Lambda_{n-1} = \{\vz\in \mathcal{C}(\mathbb{S}^{1},\mathbb{CP}^{n-1})| \vz(0) = \vz(2\pi) \}
\end{align*}
As earlier, we write $\Zb = (z_1,z_2,...,z_n) \in\mathbb{S}^{2n-1}$, and $\vz =[z_1:z_2:...:z_n]\in \mathbb{CP}^{n-1} $. The restriction of $\tilde{\sigma}$ to $\mathbb{S}^{2n-1}$ induces a natural symmetry on $\mathbb{CP}^{n-1}$. The above circular permutation $\tilde\sigma$ induces a map
\begin{align}
\label{Notion: sigma1}
\sigma_1:  &\mathbb{CP}^{n-1} \rightarrow \mathbb{CP}^{n-1},\quad [z_1:z_2:...:z_n] \xrightarrow{\sigma_1} [z_n:z_1:...:z_{n-2}:z_{n-1}],
\end{align}
letting the following diagram commute:

\[
\begin{tikzcd}
\label{Diag: Permu}
\tag{P1}
\mathbb{CP}^{n-1} \arrow{r}{{\sigma_1}} \arrow[swap]{d}{\psi} &  \mathbb{CP}^{n-1} \arrow[swap]{d}{\psi}  \\%
\mathbb{S}^{2n-1} \arrow{r}{\tilde{\sigma}}& \mathbb{S}^{2n-1}. 
\end{tikzcd}
\]
Here $\psi$ is a section of the fiber bundle $\mathbb{S}^{2n-1}\rightarrow \mathbb{CP}^{n-1}$. The diagram is well defined since it does not depend on a particular $\psi$ thus chosen.
We can then define the loop transformation
\begin{align*}
g_1: \Lambda_{n-1} \rightarrow \Lambda_{n-1}, \quad (g \vz)(t)= \sigma_1 \vz(\tau^{-1}t)
\end{align*}

\begin{Def}
We call a loop $\mathbf{z}(t) \in \Lambda_{n-1}$ a \textbf{reduced choreography}, if 
\begin{align}
g_1 \mathbf{z} = \mathbf{z},\text{ i.e. } \vz(t+\frac{2\pi}{n})=\sigma_1 \mathbf{z}(t) \quad \forall t\in [0,2\pi]
\end{align}
A reduced choreography is called \textbf{nontrivial} if it is not a constant in $\mathbb{CP}^{n-1}$.
\end{Def}

\subsubsection{Centred Reduced Choreography in $\mathbb{CP}^{n-2}$}
It is also possible to define an induced choreographic symmetry in loop space of $\mathbb{CP}^{n-2}$, with the help of Lim's coordinate transformation. Denote 
\begin{align*}
\Lambda_{n-2} = \{\mathbf{w}\in \mathcal{C}(\mathbb{S}^{1},\mathbb{CP}^{n-2})| \mathbf{w}(0) = \mathbf{w}(2\pi) \}
\end{align*}
We recall a following version of Lim's coordinate \cite{lim1989canonical}, adapted to our situation. 
\begin{Thm}[\cite{lim1989canonical}]
Let $Z=(z_1,z_2,...,z_n)\in \mathbb{C}^{n}$. There exists a linear transformation $f: \mathbb{C}^{n}\rightarrow  \mathbb{C}^{n}$
\begin{align}
W = f(Z), \quad W=(w_1,w_2,...,w_n)
\end{align}
s.t. 
\begin{enumerate}
\item $f$ is unitary, i.e., $f\in U(n)$
\item $\displaystyle w_n= \frac{1}{n}\sum_{1\leq i\leq n}z_i$
\end{enumerate}
\end{Thm}
Let $i_n:\mathbb{C}^{n-1} \hookrightarrow \mathbb{C}^{n} $ be the natural embedding, s.t.
\begin{align*}
(w_1,w_2,...,w_{n-1})\xrightarrow{i_n} (w_1,w_2,...,w_{n-1},0)
\end{align*}
Given a point $w=[w_1:w_2:...:w_{n-1}] \in \mathbb{CP}^{n-2}$, we proceed as the following. We first lift it to $(w_1,w_2,...,w_{n-1})\in \mathbb{S}^{2n-3}$, next we augment it to $W= i_n((w_1,w_2,...,w_{n-1}))= (w_1,w_2,...,w_{n-1},0)\in \mathbb{S}^{2n-1}$. Since symplectic transformation is invertible, let $Z=f^{-1}(W)\in \mathbb{S}^{2n-1}$, we can now play the usual permutation $\tilde{\sigma}$ on $Z$. As a result, we define the permutation map on $\mathbb{CP}^{n-2}$ as
\begin{align}
\sigma_{2}: \mathbb{CP}^{n-2}\rightarrow \mathbb{CP}^{n-2}
\end{align}
s.t. the following diagram commute

\[ \begin{tikzcd}
\label{Diag: Permu2}
\tag{P2}
\mathbb{CP}^{n-2} \arrow{r}{\sigma_2} \arrow[swap]{d}{\psi}  &\mathbb{CP}^{n-2}  \arrow[swap]{d}{\psi} \\
\mathbb{S}^{2n-3}  \arrow[swap]{d}{i_n} &  \mathbb{S}^{2n-3} \arrow[swap]{d}{i_n}\\
\mathbb{S}^{2n-1}  \arrow[swap]{d}{f^{-1}} &  \mathbb{S}^{2n-1} \arrow[swap]{d}{f^{-1}}\\
\mathbb{S}^{2n-1} \arrow{r}{\tilde{\sigma}} &  \mathbb{S}^{2n-1} 
\end{tikzcd}
\]
Here $\psi$ is a section of the bundle $\mathbb{S}^{2n-3}\rightarrow \mathbb{CP}^{n-2}$. Since rotation and translation commute when $P(\mathbf{Z})= Q(\mathbf{Z})= 0$, it follows that the diagram does not depend on the choice of $\psi$, hence $\sigma_2$ is well-defined.
Similarly we define the loop transformation
\begin{align*}
g_2: \Lambda_{n-2} \rightarrow \Lambda_{n-2}, \quad (g_2 \mathbf{w})(t)= \sigma_2 \mathbf{w}(\tau^{-1}t)
\end{align*}
\begin{Def}
We call a loop $\mathbf{w}(t) \in \Lambda_{n-2}$ a \textbf{centred reduced choreography}, if 
\begin{align}
g_2 \mathbf{w} = \mathbf{w},\text{ i.e. } \mathbf{w} (t+\frac{2\pi}{n})=\sigma_2 \mathbf{w}(t) \quad \forall t\in [0,2\pi]
\end{align}
A reduced centred choreography is called \textbf{nontrivial} if it is not a constant in $\mathbb{CP}^{n-2}$.
\end{Def}

\subsubsection{Relative Choreography in $\mathbb{R}^{2n}$}
The reduced choreographic loops defined in the last sub-section could lift to orbits in the original phase space. If $\mathbf{z}(t)\in \Lambda_{n-1}$ is a reduced choreographic loop and let $\Zb$ be its lifting to $\Lambda$. Then there exists a rotation $g\in SO(2)$ of angle $\alpha$ s.t. $\mathbf{Z}(0) = g\tilde{\sigma}\mathbf{Z}(\frac{2\pi}{n})$. Take a any frame of reference which rotates continuously (possibly non-uniformly) by the angle $\alpha$ during a time interval of length $2\pi/n$, and then continue the rotation of the frame by making its rotation velocity $2\pi/n$-periodic. Then $Z$ is simple choreographic in this frame, thanks to the rotational invariance of the Hamiltonian and to the uniqueness of integral curves through a point. We thus define the following objects:
\begin{Def}
We call a curve $\mathbf{Z}(t) \in \mathcal{C}([0,2\pi], \mathbb{R}^{2n})$
\begin{itemize}
\item a \textbf{relative choreography}, if $\mathbf{Z}(t)$ is a lifting of a reduced choreography $\mathbf{z}(t)\in \Lambda_{n-1}$.
$\mathbf{Z}(t)$ is \textbf{non-trivial} if $\mathbf{z}(t)$ is non-trivial; 
\item a \textbf{centred relative choreography}, if $\mathbf{Z}(t)$ is a lifting of a centred reduced choreography $\mathbf{w}(t)\in \Lambda_{n-2}$.
$\mathbf{Z}(t)$ is \textbf{non-trivial} if $\mathbf{w}(t)$ is non-trivial; 
\end{itemize}
\end{Def}
From now on, to simplify the symbols and discussion, we make the following convention. Let $k\in \{n-1,n-2\}$ and $\sigma: \mathbb{CP}^{k}\rightarrow \mathbb{CP}^{k}$ s.t.
\begin{align}
\sigma \mathbf{z} =
\begin{cases}
\sigma_1 \mathbf{z},\quad \text{if $\mathbf{z}\in \mathbb{CP}^{n-1}$} \\
\sigma_2 \mathbf{z},\quad \text{if $\mathbf{z}\in \mathbb{CP}^{n-2}$}
\end{cases}
\end{align}
By considering the standard symplectic structure $\Omega =\sum_{i=1}^{n}dp_i \wedge dq_i$, we see clearly that $\tilde{\sigma}: \mathbb{R}^{2n} \rightarrow \mathbb{R}^{2n}$ is a symplectic transformation, i.e. $\tilde{\sigma}^{*}\Omega = \Omega$. Now consider the natural symplectic form $\omega$ induced on $\mathbb{CP}^k$.

\begin{Lem}
\label{Lem: Sp}    
The map $\sigma: \mathbb{CP}^{k} \rightarrow \mathbb{CP}^{k}$ is both holomorphic and symplectic. 
\end{Lem}

\begin{proof}
We prove that $\sigma$ is holomorphic and symplectic in details for $k=n-1$. Similar argument works for $k=n-2$. \\
First, $\tilde{\sigma}$ could be seen as an invertible linear transformation of $\mathbb{C}^{n}$, hence $\sigma$ is holomorphic. Next we show that $\sigma$ is a symplectic transformation. Consider
\[     
\begin{tikzcd}
\mathbb{CP}^{n-1}  &\arrow{l}{\pi}  \mathbb{S}^{2n-1}  \arrow{r}{i}  &  \mathbb{R}^{2n} \end{tikzcd}
\]  
The symplectic form $\omega$ is defined by $\pi^{*}\omega = i^{*}\Omega$, where $\pi$ is natural projection and $i$ the natural inclusion. Now consider $v_1,v_2 \in T_z \mathbb{CP}^{n-1}$, which are equivalent classes of $T\mathbb{S}^{2n-1}$ taking quotient of the symmetry. Taking thus $\mathbf{Z}\in \mathbb{S}^{2n-1}$ s.t. $\pi(\mathbf{Z}) = \vz$ and $V_1,V_2\in T_Z\mathbb{S}^{2n-1}$ be their representatives. It follows from the defining equation of $\omega$ that there exists $g \in \mathbf{SO}(2)$ s.t.
\begin{align*}
\sigma^{*}\omega(v_1, v_2) = \omega(\sigma_{*} v_1, \sigma_{*} v_2) =\Omega(\tilde{\sigma}_{*}gV_1, \tilde{\sigma}_{*} gV_2) =\tilde{\sigma}^{*}\Omega(V_1, V_2)= \Omega(V_1, V_2) = \omega(v_1,v_2)
\end{align*}
The action of $g$ is a diagonal action, and the second equality is due the diagram \eqref{Diag: Permu} while the third equality is true because the action of g is in fact a symplectic transformation. The proof for the case $k=n-2$ is similar, by using the above argument and taking into account that the Lim transformation $f: \mathbb{C}^{n}\rightarrow \mathbb{C}^{n}$ is linear and symplectic (so is $f^{-1}$). 
\end{proof}
With the choreographic symmetry thus defined, we can move on to construct holomorphic spheres with choreographic symmetry.

\subsection{Definition of Choreographic Holomorphic Sphere}
\paragraph{}
Our aim is to find non-trivial reduced choreography which are integral curves of the system~\eqref{System: H2}. Such a loop $\zb(t)\in \mathbb{CP}^{n-1}$ (resp. $\mathbf{w}(t)\in \mathbb{CP}^{n-2}$) possesses lifts $\Zb(t)$ solving the original Hamiltonian system \eqref{System: H1}; such lifts are obtained by mere quadrature, as can be checked by switching to local coordinate systems in $\R^{2}$ which are adapted to the reduction by rotations( resp. rotations and translations). These lifted orbits $\Zb(t)$ are non-trivial relative choreographies (resp. non-trivial centred relative choreographies) of the original Hamiltonian system \eqref{System: H1}.
\paragraph{}
Searching non-constant periodic solutions on a hyper-surface is closely related to the conjecture of Weinstein. The proof of this conjecture when the underlying symplectic manifold is complex projective space been done by Hofer and Viterbo \cite{hofer1992weinstein}. They studied the Hamiltonian perturbed $J$-holomorphic spheres, which satisfies a nonlinear partial differential equation (PDE) of Cauchy-Riemann type. This PDE could be seen as a zero section of a fiber bundle. Now, our original Hamiltonian is symmetric with respect to permutation of particles, and this symmetry will be heritaged by the PDE. Our aim is to take the reduced choreographic symmetry into the construction of the fiber bundle. Once this is done, Palais' principle \cite{palais1979principle} guarantees that the PDE has a symmetric weak solution, and the elliptic regularity applies to show it is a classical solution. From that point, one can continue with the analysis given in \cite{hofer1992weinstein} and conclude the existence of a reduced choreography for the Hamiltonian system. 
To this end, we will define and study holomorphic spheres having a choreographic symmetry. For a systematic investigation of $J$-holomorphic curves, we refer to \cite{gromov1985pseudo, mcduff2012j, audin2014morse}. 
\paragraph{}
Let $\RS = \mathbb{C} \cup \{\infty\} \equiv \Sbb^2$ be the the Riemann sphere and $(\mathit{M},\omega)$ be a symplectic manifold. Let $J$ be an almost complex structure calibrated by $\omega$ ($J$ and $\omega$ are also said to be compatible), meaning that the symplectic structure twisted by $J$,
\[(x,y) \mapsto \omega(x,Jy),\]
is a Riemannian metric. A holomorphic sphere in $M$ is a smooth map $u: \RS \rightarrow M $ s.t. 
\begin{align}
    J \circ Tu = Tu \circ i
\end{align}
Now in particular let $M$ be $\mathbb{CP}^k$. This is indeed a complex manifold with standard complex structure $i_0$. We denote by $J_0$ the regular\footnote{one could turn to \cite{mcduff2012j} for the regularity of $J_0$} almost complex structure induced by $i_0$. Note that by a re-parametrisation of the augmented complex plane, a holomorphic sphere, after taking a cylinder parametrisation of $\hat{\mathbb{C}}$, can be written as a map $v(s,t) = u \circ \phi$, where $\phi(s,t) = exp(s + it), -\infty \leq s \leq \infty,0\leq t< 2\pi$. Let $(\tau u)(z) = u\circ \phi(s, t+\frac{2\pi}{n})$, where $t+ \frac{2\pi}{n}$ is to be understood as $t+\frac{2\pi}{n} \mod 2\pi$. Sometimes we also denote $\tau$ by letting
\begin{align}
\tau: & \RS \rightarrow \RS, \quad z \xrightarrow{\tau} e^{i\frac{2\pi}{n}}z
\end{align}
The somehow abused notion $\tau$ should not bring any ambiguity. It is to be understand as a translation of time for $t$ varible in our cylinder parametrisation, thus coincides with the definition before.
\begin{Def}[Choreographic holomorphic sphere]
A holomorphic sphere $u$ in $\mathbb{CP}^k$ is \textbf{choreographic} if 
    \[u \circ \tau = \sigma \circ u.\]
\end{Def}

In other words, if $u$ is a choreographic holomorphic sphere in $\mathbb{CP}^{n-1}$ (resp. $\mathbb{CP}^{n-2}$), then for each fixed $s\in \R$, $z(t):= u(s,t)$ is a reduced choreographic loop (resp. a centred reduced choreographic loop). 

\subsection{Choreographic Fiber Bundle}

\subsubsection{Base Manifold}
Next given $\alpha$ a $\omega$-minimal free homotopy class\footnote{This means that 
\begin{align*}
0<\langle \omega, \alpha \rangle  = \inf \{ \langle \omega, [u] \rangle | \text{$u$ is a nonconstant $J_0$-holomorphic sphere} \}  
\end{align*}}, we
consider the Hilbert Manifold $\mathcal{B}$
\begin{align}
\mathcal{B} =\{ u\in H^{2,2}(\hat{\mathbb{C}}, \mathbb{CP}^{k}) | [u] = \alpha, u(0) = P_{0}, u(\infty)= P_{\infty}, \int_{\norm{z}\leq 1} u^{*}\omega = \langle \omega, \alpha \rangle  \}
\end{align}

\begin{Pro}
\label{Pro: Palais}
Let $G=\langle g \rangle$ be the cyclic group generated by $g$, where $g u = (\sigma\circ \tau^{-1}) u$ and let $\mathcal{B}_G = Fix_{G}(\mathcal{B})$ be the $G$-invariant subset of $\mathcal{B}$. If $\mathcal{B}_G\neq \emptyset$, then $\mathcal{B}_G$ is itself a (totally geodesic) Hilbert sub-manifold. 
\end{Pro}
\begin{proof}
According to lemma \ref{Lem: Sp} $\sigma$ is a symplectic transformation, and $\CP$ is a K{\"a}hler manifold, hence $g$   induces an isometry in the Hilbert manifold $\mathcal{B}$ and by applying Palais' principle we see that $\mathcal{B}_G$ is a totally geodesic Hilbert sub-manifold.
\end{proof}
\begin{Rmk}
\label{Rmk: TwoEndsSymmetry}
The normalization condition is satisfied because $\int_{\norm{z}\leq 1} u^{*}\omega = \int_{\norm{z}\leq 1} (g u)^{*}\omega$. By passing $s\rightarrow \pm \infty$ in the cylinder parametrisation, one sees from the definition of $\mathcal{B}$ that a necessary condition for $\mathcal{B}_G\neq \emptyset$ is that 
\begin{align}
\sigma P_0 = P_0, \quad \sigma P_{\infty}= P_{\infty}.
\end{align}
Later on in lemma \ref{Lem: Chore1} and \ref{Lem: Chore2} it will turns out that this somehow is sufficient too.
\end{Rmk}

We will take $\mathcal{B}_G$ as our base space and construct a fiber bundle on it in the usual way while take the choreographic symmetry into the frame.
\subsubsection{Choreographic Fiber and Section}
Let $X_{J_0}$ contains all the complex anti-linear map  $\phi: T_z \hat{\mathbb{C}} \rightarrow T_{v}\CP$, i.e. 

\[ \begin{tikzcd}
\label{Diag: D1}
T_z \hat{\mathbb{C}} \arrow{r}{-i} \arrow[swap]{d}{\phi} & T_z \hat{\mathbb{C}} \arrow[swap]{d}{\phi} \\%
T_{v}\CP \arrow{r}{J_0}& T_{v}\CP
\end{tikzcd} \tag{D1}
\]
Denote $X^{G}_{J_0}\subset X_{J_0}$ the subset that furthermore satisfies the condition
\[ \begin{tikzcd}
\label{Diag: D2}
T_z \hat{\mathbb{C}} \arrow{r}{d\tau} \arrow[swap]{d}{\phi} & T_{\tau z} \hat{\mathbb{C}} \arrow[swap]{d}{\phi} \\%
T_{v}\CP \arrow{r}{d \sigma}& T_{\sigma v}\CP
\end{tikzcd} \tag{D2}
\]
Here $d\tau$ and $d\sigma$ are the push-forward of tangent vector, and the commuted diagram (D2) is for being consistent with choreography.\\
For $\forall u\in \mathcal{B}$, consider the pull back fiber bundle induced by the graph map $\bar{u}(z) = (z, u(z))$, i.e.,
\[ \begin{tikzcd}
\bar{u}^{*}X_{J_0} \arrow{r} \arrow[swap]{d}{\pi} &X_{J_0} \arrow[swap]{d}{\pi} \\%
\hat{\mathbb{C}} \arrow{r}{z\rightarrow (z,u(z))}& \hat{\mathbb{C}}\times \CP
\end{tikzcd}
\]
Finally define the symmetric fiber bundle $\mathcal{E}  \rightarrow \mathcal{B}$
\begin{align}
\mathcal{E} = \bigcup_{u\in \mathcal{B}} \{u\}\times H^{1,2}(\bar{u}^{*}X_{J_0}) 
\end{align}
similarly define
\begin{align}
\mathcal{E}_G = \bigcup_{u\in \mathcal{B}_G} \{u\}\times H^{1,2}(\bar{u}^{*}X^{G}_{J_0}) 
\end{align}
\begin{Lem}
\label{Lem: Ju_symmetry}
$\bar{\partial}_{J_0} u = du + J_0\circ du \circ i $ is a smooth section of $\mathcal{E}_G \rightarrow \mathcal{B}_G$
\end{Lem}
\begin{proof}
It is well known that $\bar{\partial}_{J_0} u$ is smooth section seen as $\mathcal{E}\rightarrow \mathcal{B}$. We only need to verify that the diagram \eqref{Diag: D2} commutes when $\phi = \bar{\partial}_{J_0} u$. Actually, since $u(\tau z) = \sigma u(z)$, one sees that for $\eta \in T_z \RS$
\begin{align}
\label{E1}
d_{\tau z}u\circ d_z\tau (\eta) = d_{u(z)} \sigma \circ d_{z} u (\eta) 
\end{align}
Now since $\tau: \hat{\mathbb{C}} \rightarrow \hat{\mathbb{C}}$ and $\sigma: \CP\rightarrow \CP$ are holomorphic maps,
\begin{align}
&J_0\circ d_{\tau z}u \circ i \circ d_{z}\tau (\eta) = J_0\circ d_{\tau z}u  \circ d_{z}\tau \circ i (\eta) \label{E2}\\
&d_{u(z)} \sigma \circ J_0\circ d_{z}u \circ i (\eta) = J_0 \circ d_{u(z)} \sigma\circ d_{z}u \circ i (\eta) \label{E3}
\end{align}
Putting \eqref{E1} into right hand side of \eqref{E2} and \eqref{E3}, one sees that $d\sigma \circ \bar{\partial}_{J_0} u = \bar{\partial}_{J_0} u \circ d\tau$.
\end{proof}
This lemma justifies in particular that the zero section corresponds to the class of choreographic holomorphic spheres in our setting.\\
\subsection{Choreographic Hamiltonian Perturbation}
\subsubsection{Invariant Hamiltonian under Choreographic Symmetry} 
Having defined the action of $\sigma:\CP \rightarrow \CP$, in this sub-section, we first show that if the Hamiltonian is of the form \eqref{System: H1}, then the reduced Hamiltonian system is invariant under relative choreographic symmetry. 
\begin{Lem}
\label{Lem: In}
If $\HR2N$ is invariant under $\tilde{\sigma}$, then $H$ is invariant under $\sigma$.
\end{Lem}
\begin{proof}
First suppose that $k= n-1$. According to the diagram \eqref{Diag: Permu}, $\exists \mathbf{Z}\in \mathbb{S}^{2n-1} $ s.t.
\begin{align}
H(\vz) = \HR2N(\mathbf{Z}) = \HR2N(\tilde{\sigma}\mathbf{Z}) = H(\sigma\mathbf{z})
\end{align}
Similar argument works for the case $k=n-2$.
\end{proof}
Now since both the reduced Hamiltonian and the symplectic form on $\CP$ are invariant under the action of $\sigma$, we have proved actually that
\begin{Pro}
\label{Pro: Invariant}
Suppose that $\HR2N: \mathbb{R}^{2n} \rightarrow \mathbb{R}$ is a $\tilde{\sigma}$-invariant Hamiltonian, i.e., 
$\HR2N(\tilde{\sigma} \mathbf{Z}) = \HR2N( \mathbf{Z}),\forall \mathbf{Z}\in \mathbb{R}^{2n}$.
Then the flow of reduced Hamiltonian $\phi_H(t)$ on $\CP$ is $\sigma$-invariant, i.e.,
\begin{align*}
\phi^{t}_H(\sigma \mathbf{z}) = \sigma \phi^{t}_H (\mathbf{z}), \forall \mathbf{z}\in \CP
\end{align*}
\end{Pro} 
\begin{proof}
Direct consequence of lemma \ref{Lem: In} and lemma \ref{Lem: Sp}. 
\end{proof}

Now let $H: \CP \rightarrow \mathbb{R}$ be a smooth map satisfying 
\begin{Hyp}
\label{Hyp: Hamiltonian}
\begin{align*}
& H(\sigma z) =  H(z), \forall z \in \CP \\
& H|_{\mathcal{U}(P_{0})} = h_0 \in \mathbb{R}, H|_{ \mathcal{U}(P_{\infty})} = h_{\infty} \in \mathbb{R}\\
& h_0 < h_{\infty}, \quad h_0 \leq H \leq h_{\infty} 
\end{align*}
where $\mathcal{U}(P_{0})$ and $\mathcal{U}(P_{\infty}) $ are $\sigma$-invariant open neighbourhood of $P_{0}$ and $P_{\infty}$, respectively. 
\end{Hyp}
\begin{Rmk}
$\mathcal{U}(P_{0})$ and $\mathcal{U}(P_{\infty})$ can be assumed to be $\sigma$-invariant because $H$ is $\sigma$-invariant.
\end{Rmk}

We define $\bar{h}(z,v) :=\phi$ be the unique complex antilinear map 
\begin{align*}
\phi: T_z \hat{\mathbb{C}} \rightarrow T_v \CP, \phi(z) =\begin{cases}
0, z\in \{0,\infty\}\\
\frac{1}{2\pi}H'(v)
\end{cases}
\end{align*}
Let $h(u)(z) = \bar{h}(z,u(z))$. The following lemma shows that, if in particular $u\in \mathcal{B}_G$, then $h(u)(z)$ will respect the choreographic symmetry
\begin{Lem}
\label{Lem: H_symmetry}
Under hypothesis \ref{Hyp: Hamiltonian},  $h(u)$ is a section from $\mathcal{B}_G$ to $\mathcal{E}_G$.
\end{Lem}
\begin{proof}
Clearly $h(u)$ is in $\mathcal{E}$. Now for $z\neq 0$, let $\eta\in T_z\RS$, then there exists a unique $\lambda\in \mathbb{C}$ s.t. $\eta = \lambda z$. Since $u$ is a choreographic holomorphic sphere and that $H(u) = H(\sigma u)$, we see that
\begin{align*}
&\phi_{\tau z}(\tau (\eta)) =\phi_{\tau z}(\tau (\lambda z)) = \phi_{\tau z}((\lambda\tau  z)) = \bar{\lambda}\phi_{\tau z}(\tau z) = \bar{\lambda}d{\sigma}( \phi_{z}(z) ) = d{\sigma}( \bar{\lambda}\phi_{z}(z) ) = d{\sigma}( \phi_{ z}( \lambda z) )
\end{align*}
where the fourth equality is due to proposition \ref{Pro: Invariant}, i.e., 
\begin{align*}
\phi_{\tau z}(\tau z) = \frac{1}{2\pi}\nabla H(u(\tau z)) = \frac{1}{2\pi}\nabla H(\sigma u(z)) = \frac{1}{2\pi}d\sigma\nabla H(u(z)) = d\sigma(\phi_{z}(z) )
\end{align*}
In other words, we have verified that if $u\in \mathcal{B}_G$ then $h(u)\in \mathcal{E}_G$.
\end{proof}

Our aim is to study the parameter depending family of smooth sections defined by 
\begin{align*}
f_{\lambda}(u) = \bar{\partial}_{J_0} u + \lambda h(u)
\end{align*}
\begin{Pro}
Under hypothesis \ref{Hyp: Hamiltonian}, $f_{\lambda}(u)$ is a section of $\mathcal{E}_{G} \rightarrow \mathcal{B}_G$
\end{Pro}
\begin{proof}
Direct consequence of lemma \ref{Lem: Ju_symmetry} and lemma \ref{Lem: H_symmetry}.
\end{proof}

\begin{Rmk}
Note that in general, for $u\notin \mathcal{B}_G$ or $H$ that is not $\sigma$-invariant (hence $h(u)$ is no longer a section from $\mathcal{B}_G$ to $\mathcal{E}_G$.) $f_{\lambda}(u)$ can still be seen a section $\mathcal{E}\rightarrow \mathcal{B}$. 
\end{Rmk}
We define moreover the sets of pairs
\begin{align}
&\mathcal{C} = \{ (\lambda,u) \in \mathbb{R}^{+} \times \mathcal{B}\text{ }|\text{ } f_{\lambda}(u) = 0\}\label{Cset} \\
&\mathcal{C}_G = \{ (\lambda,u) \in \mathbb{R}^{+} \times \mathcal{B}_{G}\text{ }|\text{ } f_{\lambda}(u) = 0\} \label{CGset}
\end{align}
We will also denote by $ \mathcal{C}(\lambda)$ a slice of $\mathcal{C}$, and $ \mathcal{C}_G(\lambda)$ a slice of $\mathcal{C}_G$, i.e.
\begin{align}
\mathcal{C}(\lambda) = \{u \in\mathcal{B}\text{ }|\text{ }
 f_{\lambda}(u) = 0\}\quad \mathcal{C}_G(\lambda) = \{u \in\mathcal{B}_{G}\text{ }|\text{ }
 f_{\lambda}(u) = 0\}
\end{align}
In particular, $\mathcal{C}(0)$ is the set of normalised holomorphic spheres of homotopy class $\alpha$ with two ends in $P_{0}$ and $P_{\infty}$. We show next that when $P_{0}$ and $P_{\infty}$ are chosen to be two special points, one has $\mathcal{C}(0)$ = $\mathcal{C}_G(0)$.
\\
\subsection{Well Posedness of Choreographic Holomorphic Sphere}
So far we have constructed $\mathcal{B}_G$, $\mathcal{C}_G$ in an abstract manner, yet we have not answered some essential questions. For example, are there non-empty choreographic holomorphic spheres, i.e., whether $\mathcal{C}_G(0)$ is not empty? In this sub-section we distinguish the two cases when $k = n-1$ and $k=n-2$ relatively and we justify the well posedness of these notions by explicit calculation. It has already been mentioned in remark \ref{Rmk: TwoEndsSymmetry} that the two ends must be carefully chosen. It turns out that this is actually enough.
\subsubsection{Special configurations in $\mathbb{CP}^{n-1}$}
Let us consider two configurations in $\mathbb{CP}^{n-1}$, denoted by $A$ and $B$ respectively, such that 
\begin{align}
& A = [1:1:1:,...,1:1] \tag{total collision}\\
& B = [e^{i \frac{2\pi}{n}  }: e^{i \frac{4\pi}{n}  }:e^{i \frac{6\pi}{n}}:...:e^{i \frac{2\pi}{n}(n-1)  } :1 ] \tag{n-polygon}
\end{align}
We call A the \textbf{total collision configuration}, and B the \textbf{n-polygon configuration}. Note that they are both $\sigma$-invariant. Assume that $P_{0} = B$ and $P_{\infty}=A$,
\begin{Lem}
\label{Lem: Chore1}
All the simple holomorphic sphere $u: \hat{\mathbb{C}} \rightarrow \mathbb{CP}^{n-1}$ s.t.   $u(0) = B$ and $u(\infty) = A$ are choreographic holomorphic sphere. 
\end{Lem}
\begin{proof}
Consider $\RS$ with the complex projective line $\mathbb{CP}^1$ by identifying $z\in \RS$ with $[z:1] \in \mathbb{CP}^1$. Suppose that $[\eta_A : \eta_B] = [z:1]$, and define a holomorphic sphere $u : \RS \rightarrow \CP$ by 
\begin{align} 
u(z) = u([\eta_A: \eta_B]) =[
\eta_A + \eta_Be^{i \frac{2\pi}{n}  } : 
\eta_A + \eta_Be^{i \frac{4\pi}{n}  } : 
\eta_A + \eta_Be^{i \frac{6\pi}{n}  } : ...: 
\eta_A +\eta_B
]
\end{align}
By explicit calculation
\[
\begin{cases}
    u(0) = u([0:1]) = B\\
    u(\infty) = u([1:0]) = A.
\end{cases}
\]
Then for $-\infty < r <\infty$,
\begin{align*}
u(\tau z) &= u(exp(r + i(t+\frac{2\pi}{n})) = u (exp(i \frac{2\pi}{n})z)
= u([e^{i\frac{2\pi}{n}}\eta_A:\eta_B])\\
&= [
\eta_A e^{i \frac{2\pi}{n}  }+ \eta_B e^{i \frac{2\pi}{n}  }  :  
\eta_A e^{i \frac{2\pi}{n}  }+ \eta_B e^{i \frac{4\pi}{n}  }  :  
\eta_A e^{i \frac{2\pi}{n}  }+ \eta_B e^{i \frac{6\pi}{n}  }  : ...: 
\eta_A e^{i \frac{2\pi}{n}  }+ \eta_B ]\\
&=\sigma u(z)
\end{align*}
As a result, $g u = u$. Next, suppose that $v: \RS \rightarrow \CP $ is another simple holomorphic sphere running through $A$ and $B$ of the same homotopy class. By calculate the Gromov-Witten invariant if necessary (see for example \cite[chapter 7]{mcduff2012j}), one sees that $v(\hat{\mathbb{C}}) = u(\hat{\mathbb{C}})$, as a result there exists then a M{\"o}bius transformation $\phi: \RS \rightarrow \RS$ s.t. $v(z) = u(\phi(z))$ and $v(0) = B, v(\infty) = A$, it follows that $v(z) = u(\zeta z) $ for some non-zero $\zeta \in \mathbb{C}$. This implies 
\begin{align}
\tau v(z) = \tau u(\zeta z) = u(\tau \zeta z) = \sigma u(\zeta z) = \sigma v(z)
\end{align}
Hence $v$ is clearly choreographic.
\end{proof}

\subsubsection{Special Configurations in $\mathbb{CP}^{n-2}$}
When it comes to the case $V=cst$ in system \eqref{System: H1}, the reduced phase space is $\mathbb{CP}^{n-2}$. The situation is slightly more complicated. We cannot use the total collision point any longer, because $P(\mathbf{Z}) = Q(\mathbf{Z}) = 0$ and $z_i = z_j,1\leq i<j\leq n$ implies that $\mathbf{Z} = 0$. Thus the total collision configuration does not exist on $\mathbb{CP}^{n-2}$. On the other hand, if we give up the reduction of translation, we cannot exclude the triviality later on (this point will become more clear in section \ref{Section: Application}).

In this sub-section we make an extra assumption that \textbf{$n=2m$ is an even integer}.
Let us consider two points $\mathbf{Z}_A$ and $\mathbf{Z}_B$ in $\mathbb{R}^{2n}$ s.t.
\begin{align}
& \mathbf{Z}_{A} = (e^{i \frac{2\pi}{m}  }, e^{i \frac{4\pi}{m}  },...,1,e^{i \frac{2\pi}{m}  }, e^{i \frac{4\pi}{m}  },...,1) \\
& \mathbf{Z}_{B} = (e^{i \frac{2\pi}{n}  }, e^{i \frac{4\pi}{n}  },e^{i \frac{6\pi}{n}},...,e^{i \frac{2\pi(n-1) }{n} } ,1 ) 
\end{align}
Note that these two points are centred, hence after Lim's coordinate transformation $\mathbf{W}=f(\mathbf{Z})$, they become two points 
\begin{align}
&\mathbf{W}_A=(w_1(\mathbf{Z}_{A}),w_2(\mathbf{Z}_{A}),....,w_{n-1}(\mathbf{Z}_{A}),(0,0)) \in \mathbb{R}^{2n}\\
&\mathbf{W}_B=(w_1(\mathbf{Z}_{B}),w_2(\mathbf{Z}_{B}),....,w_{n-1}(\mathbf{Z}_{B}),(0,0))\in \mathbb{R}^{2n}
\end{align}
They thus pass to two configurations in $\mathbb{CP}^{n-2}$, denoted as $A$ and $B$
\begin{align}
& A = [w_1(\mathbf{Z}_{A}):w_2(\mathbf{Z}_{A}):...:w_{n-1}(\mathbf{Z}_{A})] \tag{binary total collision}\\
& B = [w_1(\mathbf{Z}_{B}):w_2(\mathbf{Z}_{B}):...:w_{n-1}(\mathbf{Z}_{B})] \tag{n-polygon}
\end{align}
We call A the \textbf{binary total collision configuration}, and B the \textbf{n-polygon configuration}. Assume that $P_{0} =A $ and $P_{\infty}= B$. 
\begin{Lem}
\label{Lem: Chore2}
All the simple holomorphic sphere $u: \hat{\mathbb{C}} \rightarrow \mathbb{CP}^{n-2}$ s.t.   $u(0) = A$ and $u(\infty) = B$ are choreographic holomorphic sphere. 
\end{Lem}
\begin{proof}
Consider $\RS$ with the complex projective line $\mathbb{CP}^1$ by identifying $z\in \RS$ with $[z:1] \in \mathbb{CP}^1$. Suppose that $[\eta_B : \eta_A] = [z:1]$, and define a holomorphic sphere $u : \RS \rightarrow \mathbb{CP}^{n-2}$ by 
\begin{align} 
u(z) = u([\eta_B: \eta_A]) =[ 
&\eta_B w_1(\mathbf{Z}_{B})   + \eta_A w_1(\mathbf{Z}_{A}) : \notag\\
&\eta_B w_2(\mathbf{Z}_{B})   + \eta_A w_2(\mathbf{Z}_{A}) : \notag\\
&\eta_B w_3(\mathbf{Z}_{B})   + \eta_A w_3(\mathbf{Z}_{A}) : ...:\notag\\
&\eta_B w_{n-1}(\mathbf{Z}_{B})   + \eta_A w_{n-1}(\mathbf{Z}_{A})]
\end{align}
By the definition of $\sigma$, we see that 
\begin{align}
u(\tau z) = u([&e^{i \frac{2\pi}{n}  }\eta_B: \eta_A])\\
=[ 
&e^{i \frac{2\pi}{n}  }\eta_B w_1(\mathbf{Z}_{B})   + \eta_A w_1(\mathbf{Z}_{A}) :...:e^{i \frac{2\pi}{n}  }\eta_B w_{n-1}(\mathbf{Z}_{B})   + \eta_A w_{n-1}(\mathbf{Z}_{A})] \notag\\
=[ 
& w_1(e^{i \frac{2\pi}{n}  }\eta_B \mathbf{Z}_{B}  + \eta_A \mathbf{Z}_{A}): ...: w_{n-1}(e^{i \frac{2\pi}{n}  }\eta_B \mathbf{Z}_{B}  + \eta_A \mathbf{Z}_{A})]
\end{align}
Now one verifies easily that 
\begin{align*}
e^{i \frac{2\pi}{n}}\eta_B \mathbf{Z}_{B} + \eta_A \mathbf{Z}_{A} = 
&( \eta_B e^{i \frac{4\pi}{n}  }  + \eta_A e^{i \frac{2\pi}{m}}, \eta_B e^{i \frac{6\pi}{n}  }  + \eta_A e^{i \frac{4\pi}{m}} ,\\
&..., \eta_B e^{i \frac{2(m+1)\pi}{n}  } +\eta_A, \eta_B e^{i \frac{2(m+2)\pi}{n}  }   +\eta_A e^{i \frac{2\pi}{m}},\\
&..., \eta_B + \eta_A e^{i \frac{(n-1)\pi}{m}}, e^{i \frac{2\pi}{n}}\eta_B+ \eta_A )\\
= &e^{i \frac{2\pi}{m}} \tilde{\sigma}(\eta_B \mathbf{Z}_{B} + \eta_A \mathbf{Z}_{A})
\end{align*}
Thus 
\begin{align}
u(\tau z)  =[ 
& w_1(e^{i \frac{2\pi}{m}} \tilde{\sigma}(\eta_B \mathbf{Z}_{B} + \eta_A \mathbf{Z}_{A})
): ...: w_{n-1}(e^{i \frac{2\pi}{m}} \tilde{\sigma}(\eta_B \mathbf{Z}_{B} + \eta_A \mathbf{Z}_{A})]\\
= [ 
& w_1(\tilde{\sigma}(\eta_B \mathbf{Z}_{B} + \eta_A \mathbf{Z}_{A})
): ...: w_{n-1}(\tilde{\sigma}(\eta_B \mathbf{Z}_{B} + \eta_A \mathbf{Z}_{A})] \notag
\end{align}
Now by the definition of the action $\sigma$ for $\mathbb{CP}^{n-2}$ (see diagram \eqref{Diag: Permu2}), one sees that 
\begin{align}
u(\tau z) = \sigma u(z)      
\end{align}
The rest of the proof is the same as that in lemma \ref{Lem: Chore1}.
\end{proof}

\subsubsection{The Compactness of $\mathcal{C}_G(0)$}
Let $\alpha$ be the $\omega$-minimal class, and 
\begin{align}
\mathcal{H}(\alpha, J_0, P_0, P_{\infty}) = \{&u\in\mathcal{C}^{\infty} [\hat{\mathbb{C}},\mathbb{CP}^{k}] | \notag\\
&[u] = \alpha,\notag \\
&u(0) = P_{0} \in \mathbb{CP}^{k},\notag\\
&u(\infty)= P_{\infty}\in \mathbb{CP}^{k},\notag \\ 
&\int_{\norm{z}\leq 1} u^{*}\omega = \langle \omega, \alpha \rangle, \notag\\
&\bar{\partial}_{J_0}u =0 \} 
\end{align}
Then lemma \ref{Lem: Chore1} and \ref{Lem: Chore2} actually imply that 
\begin{Pro}
\label{Pro: Compact0slice}
Let $P_{0} , P_{\infty} $ be chosen as in lemma \ref{Lem: Chore1} and in lemma \ref{Lem: Chore2} respectively, and let $\mathcal{H}(\alpha, J_0, P_0, P_{\infty})$ be defined as above. Then $\mathcal{C}_G(0)$ is a $\mathbb{S}^1$-invariant compact manifold.
\end{Pro}
\begin{proof}
By lemma \ref{Lem: Chore1} and lemma \ref{Lem: Chore2} one sees that for such specially chosen configurations 
\begin{align}
\mathcal{C}(0) = \mathcal{C}_{G}(0) = \mathcal{H}(\alpha, J_0, P_{0}, P_{\infty})
\end{align}
It is well known that $J_0$ is regular and $\mathcal{C}(0)$ is a $\mathbb{S}^1$-invariant compact manifold for arbitrary $P_{0} \neq P_{\infty}$. The consequence follows.
\end{proof}

\subsubsection{The Non-Compactness of $\mathcal{C}_G$}
\paragraph{}
Perhaps the most crucial obeservation in the work of Hofer and Viterbo in \cite{hofer1992weinstein} is the non-compactness of $\mathcal{C}$, if $[\mathcal{H}]$, the free $\mathbb{S}^1$ cobordism class of the manifold $\mathcal{H}$, were not empty. This together with some asymptotic estimation and the Gromov compactness will permit one to find a periodic solution of the Hamiltonian system \eqref{System: H2}, although not necessarily a choreography. In our case, we can adapte ourselves to the settings in \cite{hofer1992weinstein} to show the following proposition:
\begin{Pro}
$\mathcal{C}_G$ is not compact.
\end{Pro}
\begin{proof}
See appendix \ref{Appendix: non-compactness}
\end{proof}

Finally, once a solution in the symmetric invariant manifold is found, Palais' principle then indicates that it is indeed a symmetric solution in the original manifold. After using the elliptic regularity, we see that these solutions are all smooth and they become solutions in classical sense. The estimate for asymptotic behavior of the action functional around $P_{0}$ and $P_{\infty}$ and the Gromov compactness are thus still valid as when no symmetry is involved. In particular we have actually achieve the following result, which is an choreographic analogue of \cite[Theorem 1.1]{hofer1992weinstein}:
\begin{Thm}
\label{Thm:2poleperiodic}
Let $H: \CP \rightarrow \mathbb{R}$ be a smooth Hamiltonian satisfying:\\
1. $H(\sigma z) = H(z), \forall z\in \CP$ \\
2. There exist $\sigma$-invariant open neighborhoods $\mathcal{U}(P_0)$ and $\mathcal{U}(P_{\infty})$ s.t. \\
$$H|_{\mathcal{U}(P_0)} = h_0 \in \mathbb{R}, H|_{ \mathcal{U}(P_{\infty})} = h_{\infty} \in \mathbb{R}$$\\
3. $h_0 < h_{\infty}, \quad h_0 \leq H \leq h_{\infty} $ \\

Then the Hamiltonian system $\dot{z} = \mathbf{X}_H(z)$ possesses a non-constant T-periodic reduced choreography $z^{*}$, satisfying 
\begin{align*}
h_0 < H(z^{*}) < h_{\infty}, \quad T(h_{\infty}-h_{0}) < \pi
\end{align*}
\end{Thm}
\begin{proof}
See \cite[Theorem 1.1]{hofer1992weinstein}.
\end{proof}

\subsection{Reduced Choreography}
In this sub-section let us consider the induced Hamiltonian system \eqref{System: H2} on $\CP$. When $k=n-2$, we will assume in addition that $n$ is even. Our aim is to show that, under mild conditions, the energy levels on which there exists at least one non-trivial reduced choreography form a dense set. Consider the Hamiltonian system
\begin{align*}
\dot{\vz}(t) = X_{H}(\vz(t)) = \mathcal{J}\nabla H(\vz(t)),\quad \vz\in \CP
\end{align*}
and let $P_{0},P_{\infty}$ be chosen as in lemma \ref{Lem: Chore1} and \ref{Lem: Chore2}.
\begin{Hyp}
Assume that the reduced Hamiltonian $H$ satisfies the following assumptions:
\begin{align}
&\text{$H$ is smooth} \label{V0}\tag{V0};\\
&\text{$H$ is $\sigma$-invariant, i.e. }H(\vz) = H(\sigma \vz),\forall \vz \in \mathbb{R}^{2n}\label{V1}\tag{V1};\\
& H(P_0) \neq H(P_{\infty}) \label{V2}\tag{V2}
\end{align}
\end{Hyp}
As an application of Theorem \ref{Thm:2poleperiodic} we prove the following results:
\begin{Thm}
\label{Thm: ChoreOnH}
Suppose that $H$ satisfies \eqref{V0}-\eqref{V2}. Let $\beta_1 = \min\{H(P_0),H(P_{\infty})\}, \beta_2 = \max\{H(P_0),H(P_{\infty})\}$, $\mathcal{I} = (\beta_1,\beta_2)$ be the open interval. Denote
\begin{align*}
&\mathcal{D} = \{ c\in \mathcal{I}| \text{ $S_c = H^{-1}(c)$ has a $\sigma$-invariant connected component $S_c^{\sigma}$} \}\\
&\mathcal{G} = \{c\in \mathcal{I} | \text{ $S_c = H^{-1}(c)$ possedes a reduced simple choreography on it} \}
\end{align*}
Then $\mathcal{G}$ is dense in $\mathcal{D}$.

\end{Thm}
\begin{proof}
First, let us assume in plus that $H(P_0)< H(P_{\infty})$.  If $S_c$ supports a reduced simple choreography $\mathbf{z}^c$, then $S_c$ must have a $\sigma$-invariant component, because $\mathbf{z}^c$ is $\sigma$-invariant. As a result, $\mathcal{G}\subset \mathcal{D}$. From now on suppose that $\mu(\mathcal{D})>0$. By Sard-Smale theorem, the regular value $\mathcal{R}$ is open dense in $\mathcal{I}$. Let $\mathcal{D}^{*} = \mathcal{R} \cap \mathcal{D}$. We prove next that $\mathcal{D}^{*}$ is dense in $\mathcal{G}$. Take a number $c\in \mathcal{D}^{*}$ and consider $S_c = H^{-1}(c)$. Since $\CP$ is a compact manifold and $\mathbb{R}$ is Hausdorff, \eqref{V0} implies that $H$ is a proper map. As a result, $S_c$ is compact, so is $S^{\sigma}_c$.
We can construct for small $\epsilon > 0$ a one parameter family of the form 
\begin{align*}
U_{\epsilon} = \bigcup_{\delta\in(-\epsilon,\epsilon) }S^{\sigma}_{c+{\delta}}
\end{align*}
s.t. $U_{\epsilon}$ is diffeomorphic to a sub-manifold of $\mathbb{CP}^{k}$, moreover $\sigma U_{\epsilon} = U_{\epsilon}$ because $S^{\sigma}_{c+{\delta}}$ are all $\sigma$-invariant. Note also that $U_{\epsilon}$ separate $\mathbb{CP}^k$ into two disjoint component $U_0$ and $U_{\infty}$, s.t. $P_{0}\in U_0$ and $P_{\infty} \in U_{\infty}$ (by Alexander duality).

Now by choosing  
a smooth function $\phi: \mathbb{R} \rightarrow \mathbb{R}$ s.t. 
\begin{align}
&\phi(s) = 
\begin{cases}
0 \quad \text{if $s\leq -\frac{1}{2}$};\\
1 \quad \text{if $s\geq \frac{1}{2}$}
\end{cases}\\
&\phi'(s) >0, \text{ for } s\in (-\frac{1}{2}, \frac{1}{2})
\end{align}
and let 
\begin{align}
F(\mathbf{z}) = 
\begin{cases}
\phi(\frac{\delta}{\epsilon}) \quad \text{if $\mathbf{z}\in S^{\sigma}_{c+{\delta}}$};\\
0 \quad \text{if $\mathbf{z}\in U_0 \setminus S^{\sigma}_{c+{\delta}}$}\\
1 \quad \text{if $\mathbf{z}\in U_{\infty} \setminus S^{\sigma}_{c+{\delta}}$}
\end{cases}\\
\end{align}
Since $U_{\epsilon}$ is $\sigma$-invariant, $F(\mathbf{z}) = F(\sigma\mathbf{z})$. One verifies that $F(\mathbf{z})$ satisfies the condition of theorem \ref{Thm:2poleperiodic}, by taking $U(P_0) = U_{0}, U(P_{\infty})=U_{\infty}, h_0 = 0, h_{\infty} =1$. 
Theorem \ref{Thm:2poleperiodic} then implies that $F(\mathbf{z})$ has a periodic solution $\mathbf{z}^*$, which is, after a reparametrisation of time, a reduced simple choreography of system \eqref{System: H2} and satisfies that $\displaystyle |H(\mathbf{z}^*)-c|< \frac{\epsilon}{2}$. As one has the right to choose $\epsilon$ arbitrarily small, we have actually shown that, given $c\in \mathcal{D}^*$, there exists a sequence of reduced simple choreographies of $\{\mathbf{z}^m(t)\}_{m\in \mathbb{N}}$ s.t. $H(\mathbf{z}^{m}) \rightarrow c$. The theorem is thus proved for $H(P_0)< H(P_{\infty})$.\\
\indent Now if $H(P_0) > H(P_{\infty})$, we can repeat the above argument to find \textbf{reversed} reduced choreographies\footnote{By saying reversed we mean that $\mathbf{z}(t) = \sigma \mathbf{z}(\tau t)$.} $\{\mathbf{\hat{z}}^m(t)\}_{m\in \mathbb{N}}$. Now by setting $z_{j}^{m}(t) = \hat{z}^{m}_{n+1-j}(t)$, one sees that $\mathbf{z}^{m}$ is a reduced choreography solving the Hamiltonian system while $H(\mathbf{z}^{m}) =H(\mathbf{\hat{z}}^{m}), m\in \mathbb{N}$. The theorem is proved.
\end{proof}

\section{Application to Identical N-Vortex Problems}
\paragraph{}In this section we discuss how to apply the theorems proved in the last section to examples raised from hydrodynamics, i.e. the n-vortex problem. 
The Hamiltonian system is of the form
\begin{align*}
\mathbf{\Gamma} \dot{\mathbf{Z}}(t) = \mathcal{J}_{\R^{2N}}\nabla \HR2N(\Zb(t))
\end{align*}
where 
\begin{align*}
\mathbf{\Gamma} = diag(\Gamma_1,\Gamma_1,\Gamma_2,\Gamma_2,...,\Gamma_n,\Gamma_n), \quad \Gamma_i>0, 1\leq i\leq n
\end{align*}
is the vorticity matrix. Again we distinguish two cases, 
\begin{align}
\HR2N(\mathbf{Z})& =-\frac{1}{4\pi} \sum_{1\leq i<j\leq n}\Gamma_i \Gamma_j\log|z_i-z_j|^2 \tag{n-vortex Euler}\\
\mathbf{Z}&\in \mathbb{R}^{2n} \setminus \Delta,\quad  1\leq i <j \leq n \notag\\
\HR2N(\mathbf{Z}) &=-\frac{1}{2}(\mu\sum_{i=1}^{n}\Gamma_i^2 \log\frac{1}{1-|z_i|^2} + \lambda \sum_{i<j}\Gamma_i \Gamma_j\log|z_i-z_j|^2)\tag{n-vortex BEC}\\ 
\mathbf{Z} &\in\underbrace{\mathbb{D}\times\mathbb{D}...\times \mathbb{D}}_{n} \setminus \Delta,\quad 1\leq i <j \leq n,\quad  \mathbb{D} =\{z\in \mathbb{R}^2, |z|<1 \} \notag
\end{align}
In \cite{celli2003distances}, Celli has proved that all the vorticities being identical is a necessary condition for the existence of simple choreography in the n-vortex system. We will show the existence of relative choreographies for identical n-vortex problems from both Euler equation and Gross–Pitaevskii equation by using theorem \ref{Thm: ChoreOnH}, thus showing that the vorticities being identical is also sufficient.\\
\indent{}Some properties of relative equilibria used in the proof in this section are summarised in appendix \ref{Appendix: n-vortex relative equilibria}.

\subsection{The N-Vortex Problem from Euler Equation}
\label{Section: Application}
\paragraph{}Let us consider the Hamiltonian system 
\begin{align*}
\HR2N(\mathbf{Z}) =-\frac{1}{4\pi} \sum_{1\leq i<j\leq n}\log|z_i-z_j|^2
\end{align*}

This system comes from the Euler equation that describes the interaction of $n$ identical vortices in the plane without boundary. Since there is no boundary, there is no potential part due to vortex-boundary interaction. As a result the system is invariant under the diagonal action of Euclidean group $SE(2)$, i.e., rotation and translation. Now by the discussion in previous sections, the reduced phase space is $\mathbb{CP}^{n-2}$ and we denote the reduced Hamiltonian by $H_{\mathbb{CP}^{n-2}}$.
According to the argument in the previous section, we see that the existence of centred relative choreographies near a prescribed energy level $c\in \mathbb{R}$ depends on the absence of critical points, the compactness and the symmetric component on the prescribed energy surfaces $S_c = H_{\mathbb{CP}^{n-2}}^{-1}(c)$. We summarise the validity of all the conditions in the following lemma: 

\begin{Lem}
\label{Lem: 3propertiesNvortex}
Let $S_c = H_{\mathbb{CP}^{n-2}}^{-1}(c)$ be an energy surface of the reduced Hamiltonian on $\mathbb{CP}^{n-2}$. Then
\begin{enumerate}
\item $S_c$ is compact ;
\item $S_c$ is regular except for at most finitely many $c$ ;
\item Let $B$ be the n-polygon configuration on $\mathbb{CP}^{n-2}$. There exists $\epsilon>0$ s.t. for $H_{\mathbb{CP}^{n-2}}(B) < c < H_{\mathbb{CP}^{n-2}}(B)+\epsilon$, $S_c$ has a $\sigma$-invariant component.
\end{enumerate}
\end{Lem}
\begin{proof}
We summarise the proof in the following three results:
\begin{itemize}
\item{\textbf{Compactness}: }This is proved in \cite[theorem 2.2]{wang2018relative};
\item{\textbf{Regular Value}: }This is proved in\cite[lemma 3.1]{wang2018relative};
\item{\textbf{Connectivity}: }This is proved in proposition \ref{Pro: Smallconnectbig} in appendix \ref{Appendix: n-vortex relative equilibria} of this article.
\end{itemize}
\end{proof}
We now apply theorem \ref{Thm: ChoreOnH} and lemma \ref{Lem: 3propertiesNvortex} to conclude that :

\begin{Thm}
\label{Thm: ChoreVortexEuler}
Consider the system \eqref{System: H1} with the Hamiltonian 
\begin{align*}
\HR2N(\mathbf{Z}) =-\frac{1}{4\pi} \sum_{1\leq i<j\leq n}\log|z_i-z_j|^2
\end{align*}
Assume that $n$ is even. Then there exist infinitely many non-trivial centred relative choreography.
\end{Thm}

\subsection{The N-Vortex Problem from Gross–Pitaevskii Equation}
Now we would like to show that there exist infinitely many relative simple choreographies by following similar lines as those in the previous example. To this end, we isolate vortices away from the boundary by choosing $I(z) = n\rho$ for $\rho<\frac{1}{n}$. Let $H_{\mathbb{CP}^{n-1}}$ be the reduced Hamiltonian on $\mathbb{CP}^{n-1}$.   

\begin{Lem}
\label{Lem: 3propertiesNvortexBEC}
Let $S_c = H^{-1}(c)$ be energy surface of the reduced Hamiltonian on $\mathbb{CP}^{n-1}$. Then
\begin{enumerate}
\item $S_c$ is compact; 
\item $S_c$ is regular except on $K\subset \mathbb{R}$ s.t. $\mu(K)=0$ and $K$ is closed.  
\item There exists $\epsilon>0$ s.t. for $H(B) < c < H(B)+\epsilon$, $S_c$ has a $\sigma$-invariant component.
\end{enumerate}
\end{Lem}
\begin{proof}
We prove the three properties one by one.
\begin{itemize}
\item{\textbf{Compactness}:} Since $\rho< \frac{1}{n}$, $I(z)<1$. In particular, $|z_i|^2< 1, \forall 1\leq i\leq n$. Thus the vortices are isolated from the boundary. Let $S_c$ be an energy surface that is non-empty. Since $I(z)<1$ the mutual distances are bounded above uniformly, hence they are also bounded below uniformly. This implies in particular that  $S_c$ is isolated not only from the boundary but also from the generalised diagonal $\Delta$(where collisions happen). As a result $S_c$ is compact.
\item{\textbf{Regular Value}:} Define 
\begin{align}
\mathcal{R} = \{ c\in \mathbb{R}| S_c \text{ is regular}  \}
\end{align}
In lemma \ref{Lem: ShubBEC} we have shown that the relative equilibria of the n-vortex problem in BEC are isolated from the generalised diagonal set $\Delta$. Hence $\mathcal{R}$ is open and dense as a direct application of the Sard-Smale theorem.
\item{\textbf{Connectivity}:} Observe that if we restrict the Hamiltonian to the set 
\begin{align*}
\mathcal{M}_{\rho} = \{|z_1|^2= |z_2|^2=\dotsm =|z_n|^2 = \rho\}, \rho>0
\end{align*} 
then the part $\mu\sum_{i=1}^{n} \log\frac{1}{1-|z_i|^2}$ becomes a constant. As a result, proposition \ref{Pro: Smallconnectbig}  applies and we are done.
\end{itemize}
\end{proof}
Note that unlike the case of the n-vortex problem from Euler equation, the n-vortex problem in BEC is not invariant under translation hence we don't know if the solution found is a centred relative choreography or not. It is possible that the solution being a configuration that vorticies rotating around its center of vorticity which is not $(0,0)$. We would like to exclude this possibility. As a result, we must convince ourselves that these orbits are not an n-polygon configuration put in a rotating frame. We prove the following proposition:
\begin{Pro}
\label{Pro: nontrivialcpn1}
Suppose that the n-vortex problem in BEC has a n-polygon relative equilibrium in some rotating frame. Then it must be centred.
\end{Pro}
\begin{Rmk}
When we say any, we mean that this rotating frame can be centred at either the origin or any other point; and the rotating speed can be either uniform or not.  
\end{Rmk}

\begin{proof}
We argue by contradiction. Suppose that the solution thus found is an n-polygon configuration in a rotational frame, then it looks like in figure \ref{fig: NonTrivial}. Let $r_1 = \norm{OO'}$ be the distance between the origin and the centre of vorticity, and $r_2 = \norm{O'A_1}= \norm{O'A_2}=...=\norm{O'A_n}$ be constant, then by the cosine formulae :
\begin{align*}
&\norm{OA_1}^2 = r_1^2 + r_2^2 - 2 r_1r_2 \cos{\theta}\notag\\
&\norm{OA_2}^2 = r_1^2 + r_2^2 - 2 r_1r_2 \cos{(\theta+\frac{2\pi}{n})}\notag\\
&...\\
&\norm{OA_n}^2 = r_1^2 + r_2^2 - 2 r_1r_2 \cos{(\theta+\frac{2(n-1)\pi}{n})}
\end{align*}
By the assumption, $\sum_{1\leq i \leq n} log (1-|z_i|^2)$ is a constant, which implies, by denoting $\alpha = 1- r_1^2 - r_2^2, \beta =  2 r_1r_2$, that the following quantity is a constant too.
\begin{align*}
cst = &(1-\norm{OA_1}^2)(1-\norm{OA_2}^2)...(1-\norm{OA_n}^2) \\
=& (\alpha + \beta \cos{\theta})(\alpha + \beta \cos{(\theta+ \frac{2\pi}{n})})(\alpha + \beta \cos{(\theta+ \frac{2(n-1)\pi}{n})})\\
=& \sum_{k=1}^{n} (a_k \cos{k\theta} +b_k \sin{k\theta}) + \alpha^n
\end{align*}
This is a trigonometric polynomial, hence to be a constant one must have $a_k = b_k = 0,1 \leq k \leq n$. In particular, explicit calculation shows that $a_n =0$ implies that $\beta^n=0$.
To make the above trigonometric polynomial a constant, we thus need that $\beta =0$. In other words, either $r_1 =0$ or $r_2 = 0$ (they cannot be both 0 because otherwise it corresponds to no point in $\mathbb{CP}^{n-1}$) However, 
\begin{itemize}
\item{$r_1=0 $:} in this case the orbit corresponds to the centred n-polygon configuration;
\item{$r_2=0 $:} in this case the orbit corresponds to the total collision configuration. 
\end{itemize}
\begin{figure}
\begin{center}
\includegraphics[width=40mm,scale=0.5]{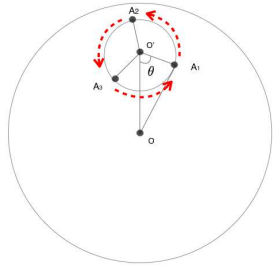}
\caption{A 3-polygon (equilateral triangle) in a rotational frame}
\label{fig: NonTrivial}
\end{center}
\end{figure}
The conclusion follows.
\end{proof}
Now we have actually proved the following theorem :
\begin{Thm}
\label{Thm: ChoreVortexGP}
Consider the system \eqref{System: H1} with the Hamiltonian 
\begin{align*}
\HR2N(\mathbf{Z}) =-\frac{1}{2}(\mu\sum_{i=1}^{n} \log\frac{1}{1-|z_i|^2} + \lambda \sum_{i<j}\log|z_i-z_j|^2)
\end{align*}
Then for any $n\in \mathbb{N}^{+}$ there exist infinitely many non-trivial relative choreography.\end{Thm}
\begin{proof}
This is a direct consequence of lemma \ref{Lem: 3propertiesNvortexBEC}, theorem \ref{Thm: ChoreOnH}, and proposition \ref{Pro: nontrivialcpn1}.
\end{proof}

\subsection{Comparison with Other Methods}
\paragraph{}
Finally we give some heuristic remarks about the solutions which could be found using perturbative methods. Let us take the 4-vortex problem  to illustrate the idea. As mentioned before, the 4 vortex problem is in general non-integrable. As a result it seems hopeless to try to describe the complete bifurcation diagram of periodic orbits. Let us consider the reduced energy Hamiltonian $\HCPN2$, and denote by $c$ the reduced energy level. There are at least two places where one might locally construct relative choreographies.

\paragraph{Bifurcation from the equilateral triangle:}
Recall that the minimum of $\HCPN2$ is achieved when the 4 vortices form the centred square (4-polygon). Thus the Moser-Weinstein theorem should show the existence of relative periodic solutions of short period, bifurcating from the square  \cite{calleja2018choreographies, carvalho2014lyapunov, borisov2004absolute};

\paragraph{Bifurcation from the simultaneous pair of double collisions:}
In contrast, when the reduced energy tends to infinity, there is a pair of vortices that become close to one another. Now, consider two vortices of vorticity $2$, located respectively at $(\pm 1,0)$ (thus forming a relative equilibrium). Next, consider replace each such vortex by a pair of close vortices of vorticity $1$, that chase one another in the cluster. At the same time the two clusters will rotate approximately as two votices would. As another illustration of the superposition principle (see the periodic orbits of Bartsch et al. \cite{bartsch2017global} and the KAM tori of Khanin \cite{khanin1982quasi}), this should prove the existence of relative periodic orbits bifurcating out of the simultaneous pair of double collisions. 

\paragraph{}
Because each time we find a one-parameter family of non-trivial choreography by perturbative method, it implies that we can continue to apply our theorem \ref{Thm: ChoreOnH} to get more relative choreographies, until the symmetry of this component is broken. As a result, we believe our global approach can be seen as producing solutions of similar interests by both perturbation around the 4-polygon and around pairs of binary collisions. 

\section*{Acknowledgement}
The author would like to thank Jacques F{\'e}joz and Eric S{\'e}r{\'e} for having introduced the subject of n-vortex problem to me and taught me many ideas in $J$-holomorphic curves and choreography, also for their consistent patience and encouragement on this topic. 
\begin{figure}
\begin{center}
\includegraphics[width=80mm,scale=0.5]{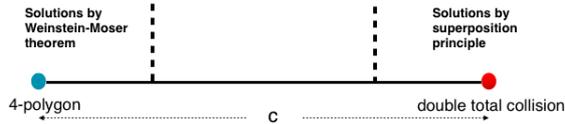}
\caption{The configuration changing with reduced energy level}
\label{Fig:Energy_Config}
\end{center}
\end{figure}
\appendix
\section{The Non-Compactness of $\mathcal{C}_G$}
\label{Appendix: non-compactness}
\paragraph{}
By adapting ourselves with the symmetric constraint taken into account, we would like to show the existence of a reduced simple choreography. To this end we need some non-compactness for the $\mathcal{C}_{G}$. More precisely,denote $[\mathcal{M}]$ the free $\mathbb{S}^1$ cobordism class of the manifold $\mathcal{M}$, in \cite{hofer1992weinstein} the following proposition is proved:
\begin{Pro}\cite[proposition 2.7]{hofer1992weinstein}
\label{Pro: HoferViterboNoCompact}
Let $(V,\omega)$ be a complex symplectic manifold and $J$ be a regular almost complex structure calibrated by $\omega$, $\alpha$ be a $\omega$-minimal free homotopy class, and $P_{0}, P_{\infty}$ are disjoint closed submanifold and $\mathcal{C}$ is defined as in \eqref{Cset}. If $\mathcal{C}$ is compact, then $[\mathcal{C}(0)] = [\emptyset]$.
\end{Pro}
The sketch of the proof can be described in the following steps:
\begin{enumerate}
\item
Equipe a neighborhood $\mathcal{W}$ of $\mathcal{C}$ with an appropriate fiber metric. 
\item
As typical way of dealing such problems, an appropriate perturbation section space $\mathcal{P}$ is carefully chosen, together with its completion $\mathcal{G}$ under certain Sobolev norm. We consider the perturbated section
\begin{align} 
&\mathcal{G}_{k,\delta}\times  \mathcal{W} \rightarrow \mathcal{E}\notag\\
&(r,\lambda, u) \rightarrow F(r,\lambda, u) = f_{\lambda}(u) + \bar{r}_{\lambda}(u)
\end{align}
here $k\in \mathbb{N}$ measures the regularity of Sobolev norm and $0<\delta$ measures the size of the completed perturbation set in such norm, and   $\bar{r}_{\lambda}(u)(z) := r(\lambda, z, u(z))$ is a perturbation section.
\item
When $k$ is large enough and $\delta$ is small enough, $F$ can be proved to be a Fredholm section.
\item $f_{\lambda}$ is already known to be a Fredholm section and thus locally proper, as a result the set 
\begin{align}
\mathcal{M} = \{(\lambda, u)\in \mathcal{W}| F(r,\lambda,u)=0  \}
\end{align}
can be proved to be a compact manifold.
\item
By replacing $r$ by an element in $\mathcal{P}$ (due to density of $\mathcal{P}$ in $\mathcal{G}$) if necessary, $\mathcal{C}(0)$ could be achieved as the boundary of a compact manifold $\mathcal{M}$.
\end{enumerate}
In our case, we note that 
\begin{Lem}
$f_{\lambda}$ seen as a section of $\mathcal{E}_G \rightarrow \mathbb{R}\times \mathcal{B}_G$ is Fredholm.
\end{Lem}
\begin{proof}
As before let $z = exp(s+it)$ be the cylinder parametrisation. Then by Palais' principle, 
\begin{align}
\mathcal{TB}_G= \{ (u(s,t), \zeta(s,t)) \in\mathcal{TB} \text{ }|
\text{ }\zeta(s, t+\frac{2\pi}{n}) = d\sigma (\zeta(s, t) )  \}
\end{align}
This implies that $Ker(df_{\lambda})$ seen as a linearised operator $\mathbb{R}\times \mathcal{B}_G \xrightarrow{df_{\lambda}} \mathcal{E}_G$ is a subspace of $Ker(df_{\lambda})$ seen as the kernel of the same linearised operator of $\mathbb{R}\times \mathcal{B} \xrightarrow{df_{\lambda}} \mathcal{E}$. Since the latter is of finite dimension, we see that the former is also of finite dimension. To see that the co-kernel is also of finite dimension, note that if $u\in \mathcal{B} \setminus  \mathcal{B}_G$, then $\exists \mu\in \mathbb{R}, \zeta(s,t) \in \mathcal{T}_u \mathcal{B}$, $ df_{\lambda}(\mu, \zeta) \notin \mathcal{E}_G$. As a result, due to that $Coker(df_{\lambda})$ seen as a linearised operator $\mathbb{R}\times \mathcal{B} \xrightarrow{df_{\lambda}} \mathcal{E}$ is of finite dimension, $Coker(df_{\lambda})$ seen as a linearised operator $\mathbb{R}\times \mathcal{B}_G \xrightarrow{df_{\lambda}} \mathcal{E}_G$ is of finite dimension too.
\end{proof}
The admissible perturbation sections that are adaptive to our choreographic symmetric constraint are defined as the following: 
\begin{Def}
Given the projection map:
\begin{align*}
P: [0,\lambda_{\infty}+1] \times \RS \times \CP \rightarrow \RS \times \CP,\quad P(\lambda,z,v) = (z,v)
\end{align*}
and the pull-back bundle $P^{*}X_J^G\rightarrow [0,\lambda_{\infty}+1] \times \RS \times \CP $. Consider vector spaces $\mathcal{A}$ of all smooth section $r(\lambda, z,v)$ of this bundle. 
\begin{align*}
&(1) \quad r(\lambda, z,v) = 0 \text{ if $\lambda$ is close to $0$ or $z$ is close to either $0$ or $\infty$.};\\ 
&(2) \quad r(\lambda, z,v) = r(\lambda, \zeta z,v),\quad  \forall \zeta \in \mathbb{S}^1.
\end{align*}
The admissible perturbation space is defined by $\mathcal{G}_k$, which is the completion of $\mathcal{A}$ in some Sobolev norm $\norm{\cdot}_{W^{k,2}}$ for some $k\in \mathbb{N}$ large enough.
\end{Def}
\begin{Rmk}
Note that then $k$ is large enough $\mathcal{G}_k$ is embedded in to continuous sections, thus the symmetric constraint is well defined for $\mathcal{G}_k$.
\end{Rmk}
Now by following the same lines in \cite{hofer1992weinstein} we can declare the following proposition.
\begin{Pro}
\label{Pro: HoferViterboNoCompactWithSymmetry}
Let $(\CP,\omega)$ be the standard complex projective space and $J_0$ be the regular almost complex structure induced by $i$. Let $P_{0}, P_{\infty}$ be chosen as in lemma \ref{Lem: Chore1} and \ref{Lem: Chore2}, $\alpha$ be the $\omega$-minimal free homotopy class that is invariant under $g$, and $\mathcal{C}_G$ is defined as in \eqref{CGset}. If $\mathcal{C}_G$ is compact, then $[\mathcal{C}_G(0)] = [\emptyset]$.
\end{Pro}
\begin{proof}
See \cite[proposition 2.7]{hofer1992weinstein} for details.
\end{proof}
Now proposition \ref{Pro: HoferViterboNoCompactWithSymmetry} and proposition \ref{Pro: Compact0slice} together indicate the non-compactness we are looking for:
\begin{Pro}
$\mathcal{C}_G$ is not compact.
\end{Pro}
\begin{proof}
By proposition \ref{Pro: Compact0slice}, we have seen that if we choose $P_{0}$ and $P_{\infty}$ in such a special way,  then $\mathcal{C}(0) = \mathcal{C}_{G}(0)$. As a result, $[\mathbb{S}^1]=[\mathcal{C}(0)]= [\mathcal{C}_{G}(0)]$. It follows by proposition \ref{Pro: HoferViterboNoCompactWithSymmetry} that $\mathcal{C}_G$ is not compact.
\end{proof}

\section{Some Properties of Relative Equilibria in The N-Vortex Problem}
\label{Appendix: n-vortex relative equilibria}
In this appendix we consider some properties of the relative equilibria in the n-vortex problem with general positive vorticity. 
\subsection{Shub's Lemma for N-Vortex in Bose Einstein Condensation}
We claim that for positive vorticities, the relative equilibria of $H$ cannot accumulated into $\Delta$. This is a version of Shub's lemma \cite{shub1971appendix} from celestial mechanics. The analogues in vortex problems without the harmonic trap are studied by \cite{o1987stationary,roberts2017morse}. The following lemma is proved by using the conservation of centre of vorticity as is in \cite{wang2018relative}.
\begin{Lem}
\label{Lem: ShubBEC}
Suppose that $\Gamma_i>0, 1\leq i\leq n$, and $\textbf{z}$ is a relative equilibrium s.t. $I(\textbf{z}(t)) = \alpha < \min_{1\leq i\leq n}\Gamma_i$. Denote 
\begin{align*}
m(\textbf{z}) = \inf_{ 1\leq i < j\leq n} |z_i-z_j|^2
\end{align*} 
then there exists a constant $\epsilon(\alpha,\mathbf{\Gamma})$ s.t.
\begin{align*}
m(\textbf{z}) > \epsilon
\end{align*}
\end{Lem}
\begin{proof}
Suppose to the contrary that $\mathbf{z}^k$ is a sequence of relative equilibria of the n-vortex problem in BEC s.t. $I(\mathbf{z}^{m}(t)) = \alpha$ and $\lim_{k\rightarrow \infty}m(\textbf{z}^k) = 0$. Then by consecutively passing to subsequence if necessary, we may suppose that there exists an sub-index set $V\subset\{1,2,..,n\}$ s.t. $z^k_i \rightarrow z^*, \forall i\in V$.  Denote $\textbf{z}_V$ as the vector of vortices with index in V. As before let $\displaystyle L = \sum_{1\leq i< j \leq n} \Gamma_i \Gamma_j$ and define moreover
$\displaystyle L_V = \sum_{\substack{i< j \\ i,j\in V}} \Gamma_i \Gamma_j$. \\
First, we show that $z^*$ cannot be an interior point inside the unit circle. Actually, observe that $c^k_V= \frac{\sum_{i\in V}\Gamma_i z^k_i}{\sum_{i\in V}\Gamma_i }$, the vorticity centre of $\textbf{z}^k_V$, also follows a uniform rotation with the vortices. Denote the angular speed to be $\nu$, then 
\begin{align}
\dot{c}^k_V &= \frac{\sum_{i\in V}\Gamma_i\dot{z}^k_i}{\sum_{i\in V} \Gamma_i} =\mathbb{J} \frac{\nu}{2} c^k_V \rightarrow \mathbb{J} \frac{\nu}{2} z^*\\
\Gamma_i\dot{z}_i &= \mathbb{J}(\nabla_{z_i} H_V(\textbf{z}) +  \nabla_{z_i} H_{V^c}(\textbf{z}) ) = \mathbb{J} \Gamma_i \frac{\nu}{2} z_i^k \rightarrow \mathbb{J} \Gamma_i \frac{\nu}{2} z^*,\quad i\in V
\end{align}
Define the vector $\displaystyle p = \lambda \sum_{j\in V^c}  \Gamma_j \frac{z^*-z_j}{\norm{z^*-z_j}^2}, q = -\mu \frac{z^*}{1-|z^*|^2}$. Thus we have 
\begin{align*}
\dot{c}^k_V& \rightarrow \frac{\sum_{i\in V}\Gamma_i^2}{\sum_{i\in V}\Gamma_i}q + p\\
\dot{z}_i^k& \rightarrow  \Gamma_i q + p + \nabla_i H_V(\textbf{z}^k)
\end{align*}
As $\dot{c}^k_V - \dot{z}_i^k \rightarrow 0$, it turns out that
\begin{align*}
\Gamma_i (\Gamma_i - \frac{\sum_{i\in V}\Gamma_i^2}{\sum_{i\in V}\Gamma_i}) q \sim\nabla_i H_V(z_k) 
\end{align*}
Hence 
\begin{align*}
-\lambda L_V = \nabla H_V(z_k) z^k \rightarrow -\mu \frac{|z^*|^2}{1-|z^*|^2}
\sum_{i\in V}\Gamma_i (\Gamma_i - \frac{\sum_{i\in V}\Gamma_i^2}{\sum_{i\in V}\Gamma_i}) =0
\end{align*}
This is impossible. As a result, $z^*$ must be a point on the boundary if it exists. But then this is also impossible because $I<\min_{1\leq i\leq n}\Gamma_i$ implies that $|z_i|^2< 1, \forall 1\leq i \leq n$. The lemma is thus proved.
\end{proof}
The above lemma implies that for positive BEC n-vortex system all the relative equilibria are bounded away uniformly from the generalized diagonal set $\Delta$. Equivalently, it means that the fixed points of the system \eqref{System: H2} on $\CP$ cannot accumulate in to $\tilde{\Delta}$, the projection of $\Delta$ on $\mathbb{CP}^{n-1}$. 
\subsection{On the Connectivity of Energy Surface of the N-Vortex Problem}
\paragraph{}In this sub-section, we show that the reduced energy surface of the n-vortex problem, both for Euler equation and Gross–Pitaevskii equation has connected surfaces that are symmetric with respect to the choreographic symmetry. \\
\paragraph{}Consider the following problem: given $n$ points $\mathbf{A} = (A_1,A_2,...,A_n)$ on the unit circle, none of them overlaps, i.e., $A_i \neq A_j,\forall 1\leq i< j\leq n$. Denotes $l_{ij} = \norm{A_iA_j}$ to be length of the segment between $A_i$ and $A_j$. We would like to consider 
\begin{align}
F(\mathbf{A}) = \prod_{1\leq i < j \leq n} \log l_{ij}
\end{align}

\begin{Lem}
\label{Lem: Npolymaximum}
$F(\mathbf{A})$ achieves its maximum when $\mathbf{A}=(A_1,A_2,...A_n)$ form a n-polygon inscribed to the unit circle.
\end{Lem}
\begin{proof}
Without loss of generality, we can assume that the index $j$ of $A_j$ increases along the clockwise direction. We can then denote by  $\theta_{j}$ the angle between $OA_j$ and $OA_{j+1}, \forall 1\leq j\leq n-1$ while $\theta_{n}$ is the angle between $OA_{n}$ and $OA_1$ (see figure \ref{Fig:4VortexBEC}). Now by the sine formule of chord length, we have that 
\begin{align}
l_{ij} = 2\sin{\frac{\theta_i+\theta_{i+1}+...+ \theta_{j-1}}{2}}, j>i
\end{align}
Note that if $j>n$, $A_j$ is to be considered as $A_{j'}$, where $j' = j \mod n$. In this way, we regroup the items in the product $F(\mathbf{A})$, such that in each subset the the difference $j-i$ is fixed. i.e., denote 
\begin{align}
F(\mathbf{A}) = \prod_{1\leq i < j \leq n} \log l_{ij} = \prod_{1\leq k  \leq [\frac{n}{2}]}B_k, \quad B_k = \prod_{\substack{1\leq i\leq n \\ j-i = k}  } \log l_{ij}
\end{align}
One verifies explicitly that $f(\theta) = \sin{\theta}, 0< \theta < \pi$ is concave, hence 
\begin{align}
B_k &= \prod_{\substack{1\leq i\leq n \\ j-i = k}} \log l_{ij}
= n\log 2+ \prod_{1\leq i\leq n}\log \sin{\frac{\theta_i+\theta_{i+1}+...+ \theta_{i+k-1}}{2}}\notag\\
&\leq  n\log 2+ n \log \sin{\sum_{1\leq i\leq n}\frac{\theta_i+\theta_{i+1}+...+ \theta_{i+k-1}}{2n}} \tag{Jensen's Inequality}\\
&= n\log 2 + n\log \sin( k\sum_{1\leq i\leq n}\frac{\theta_i}{2n})\notag\\
&= n\log 2 + n\log \sin( \frac{k\pi}{n})
\end{align}
As a result 
\begin{align}
\label{Eq: Maximaum}
F(\mathbf{A}) \leq  \prod_{1\leq k  \leq [\frac{n}{2}]}(n\log 2 + n\log \sin( \frac{k\pi}{n}))
\end{align}
Since $k\in \mathbb{N}$ and $0<k< [\frac{n}{2}]$, we see that the inequality in \eqref{Eq: Maximaum} becomes the equality if and only if
\begin{align}
\theta_1 = \theta_2 = ... =\theta_n =\frac{2\pi}{n}
\end{align}
In other words, $F(\mathbf{A})$ achieves its maximum when $\mathbf{A}=(A_1,A_2,...A_n)$ form a n-polygon inscribed to the unit circle.
\end{proof}

\begin{figure}
\begin{center}
\includegraphics[width=40mm,scale=0.5]{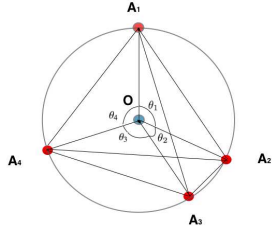}
\caption{4-vortex problem in BEC restricted to $\mathcal{M}_{\rho}$}
\label{Fig:4VortexBEC}
\end{center}
\end{figure}

Now we show the following simple yet useful lemma on the existence of a $\sigma$-invariant component of the energy surface $S_c = H^{-1}(c)$:
\begin{Lem}
\label{Lem: SymComp}
Let $S_c \subset \mathbb{CP}^{k}$ be the energy surface. Suppose that there is a connected subset $U\subset S_c$ s.t. $U$ is $\sigma$-invariant, i.e., $\sigma U = U$. If $H$ is $\sigma$-invariant, then $U$ is contained in a $\sigma$-invariant component of $S_c$.
\end{Lem}
\begin{proof}
Since $U$ is connected, $U$ is included in a component $S_c^{\sigma}$ of $S_c$. We only need to show this component itself is $\sigma$-invariant. To this end, let $\mathbf{z} \in S_c^{\sigma}$. Then $\exists \mathbf{u} \in U$ and a continuous function $f: [0,1]\rightarrow \CP$ s.t. $f(0)= \mathbf{u},f(1)=\mathbf{z}$. As a result, let $g:[0,1]\rightarrow \CP$ defined by $g(t)= \sigma f(t)$. Clearly $g$ is a continuous function satisfies that $g(0)= \sigma\mathbf{u},g(1)=\sigma\mathbf{z}$, and $\forall t\in [0,1]$, we have $H(g(t)) = H(\sigma f(t)) = H(f(t))= c$, hence $g(t) \in S_c, \forall t\in[0,1]$. We have thus shown that $\sigma\mathbf{u}$ and $\sigma\mathbf{z}$ are connected, hence $\sigma\mathbf{z}\in S_c^{\sigma}$ too. We conclude that $S_c^{\sigma}$ is the $\sigma$-invariant component and the lemma is proved.
\end{proof}

We state another useful criteria for showing that there is a symmetric component on some prescribed energy level of the reduced Hamiltonian.\\
\begin{Pro}
\label{Pro: Smallconnectbig}
Suppose that the n-polygon configuration B is a non-degenerate maximum of $H$ restricted to the manifold 
\begin{align}
\label{Def: UnitCircle}
\mathcal{M}_{\rho} = \{|z_1|^2= |z_2|^2=\dotsm =|z_n|^2 = \rho\}, \rho>0
\end{align}
Let $H(B) >c > H(B)- \epsilon$ for small $\epsilon>0$, then $S_c$ has a $\sigma$-invariant component. 
\end{Pro}
\begin{proof}
We see that by lemma \ref{Lem: Npolymaximum}, the n-polygon configuration is a maximum for $H(z)|_{\mathcal{M}{\rho}}$. 
Since it is a non-degenerate critical point, there is no other critical point nearby. As a result, the set $M_c = H^{-1}_{\mathcal{M}_{\rho}} (c)=  H^{-1} (c)\cap \mathcal{M}_{\rho}$ has a connected component, denoted as $M_c^{\sigma}$ that is $\sigma$-invariant. It is then included in a $\sigma$-invariant component $S_c^{\sigma}$, due to lemma \ref{Lem: SymComp}.
\end{proof}


\begin{thebibliography}{10}

\bibitem{aftalion2002shape}
A.~Aftalion and R.~L. Jerrard.
\newblock Shape of vortices for a rotating bose-einstein condensate.
\newblock {\em Physical Review A}, 66(2):023611, 2002.

\bibitem{audin2014morse}
M.~Audin and M.~Damian.
\newblock {\em Morse theory and Floer homology}.
\newblock Springer, 2014.

\bibitem{bartsch2016periodic}
T.~Bartsch and Q.~Dai.
\newblock Periodic solutions of the n-vortex hamiltonian system in planar
  domains.
\newblock {\em Journal of Differential Equations}, 260(3):2275--2295, 2016.

\bibitem{bartsch2017global}
T.~Bartsch and B.~Gebhard.
\newblock Global continua of periodic solutions of singular first-order
  hamiltonian systems of n-vortex type.
\newblock {\em Mathematische Annalen}, 369(1-2):627--651, 2017.

\bibitem{borisov2004absolute}
A.~V. Borisov, I.~S. Mamaev, and A.~Kilin.
\newblock Absolute and relative choreographies in the problem of point vortices
  moving on a plane.
\newblock {\em Regular and Chaotic Dynamics}, 9(2):101--111, 2004.

\bibitem{calleja2018choreographies}
R.~C. Calleja, E.~J. Doedel, and C.~Garc{\'\i}a-Azpeitia.
\newblock Choreographies in the n-vortex problem.
\newblock {\em Regular and Chaotic Dynamics}, 23(5):595--612, 2018.

\bibitem{carvalho2014lyapunov}
A.~C. Carvalho and H.~E. Cabral.
\newblock Lyapunov orbits in the n-vortex problem.
\newblock {\em Regular and Chaotic Dynamics}, 19(3):348--362, 2014.

\bibitem{castilla1993four}
M.~Castilla, V.~Moauro, P.~Negrini, and W.~M. Oliva.
\newblock The four positive vortices problem: regions of chaotic behavior and
  the non-integrability.
\newblock In {\em Annales de l'IHP Physique th{\'e}orique}, volume~59, pages
  99--115. Gauthier-Villars, 1993.

\bibitem{celli2003distances}
M.~Celli.
\newblock Sur les distances mutuelles d'une chor{\'e}graphie {\`a} masses
  distinctes.
\newblock {\em Comptes Rendus Mathematique}, 337(11):715--720, 2003.

\bibitem{chenciner2003action}
A.~Chenciner.
\newblock Some facts and more questions about the {E}ight.
\newblock In {\em Topological methods, variational methods and their
  applications ({T}aiyuan, 2002)}, pages 77--88. World Sci. Publ., River Edge,
  NJ, 2003.

\bibitem{chenciner2009unchained}
A.~Chenciner and J.~F{\'e}joz.
\newblock Unchained polygons and the n-body problem.
\newblock {\em Regular and chaotic dynamics}, 14(1):64--115, 2009.

\bibitem{chenciner2000remarkable}
A.~Chenciner and R.~Montgomery.
\newblock A remarkable periodic solution of the three-body problem in the case
  of equal masses.
\newblock {\em Annals of Mathematics-Second Series}, 152(3):881--902, 2000.

\bibitem{chris2003discrete}
J.~Chris~Eilbeck and M.~Johansson.
\newblock The discrete nonlinear schr{\"o}dinger equation—20 years on.
\newblock In {\em Localization and energy transfer in nonlinear systems}, pages
  44--67. World Scientific, 2003.

\bibitem{degiovanni1987periodic}
M.~Degiovanni, A.~Marino, and F.~Giannoni.
\newblock Periodic solutions of dynamical systems with newtonian type
  potentials.
\newblock In {\em Periodic solutions of Hamiltonian systems and related
  topics}, pages 111--115. Springer, 1987.

\bibitem{fetter2009rotating}
A.~L. Fetter.
\newblock Rotating trapped bose-einstein condensates.
\newblock {\em Reviews of Modern Physics}, 81(2):647, 2009.

\bibitem{gordon1977minimizing}
W.~B. Gordon.
\newblock A minimizing property of keplerian orbits.
\newblock {\em American Journal of Mathematics}, pages 961--971, 1977.

\bibitem{gromov1985pseudo}
M.~Gromov.
\newblock Pseudo holomorphic curves in symplectic manifolds.
\newblock {\em Inventiones mathematicae}, 82(2):307--347, 1985.

\bibitem{hofer1992weinstein}
H.~Hofer and C.~Viterbo.
\newblock The {W}einstein conjecture in the presence of holomorphic spheres.
\newblock {\em Communications on pure and applied mathematics}, 45(5):583--622,
  1992.

\bibitem{khanin1982quasi}
K.~Khanin.
\newblock Quasi-periodic motions of vortex systems.
\newblock {\em Physica D: Nonlinear Phenomena}, 4(2):261--269, 1982.

\bibitem{kirchhoff1876vorlesungen}
G.~R. Kirchhoff.
\newblock {\em Vorlesungen {\"u}ber mathematische physik: mechanik}, volume~1.
\newblock Teubner, 1876.

\bibitem{koiller1989non}
J.~Koiller and S.~P. Carvalho.
\newblock Non-integrability of the 4-vortex system: Analytical proof.
\newblock {\em Communications in mathematical physics}, 120(4):643--652, 1989.

\bibitem{lim1989canonical}
C.~C. Lim.
\newblock Canonical transformations and graph theory.
\newblock {\em Physics Letters A}, 138(6-7):258--266, 1989.

\bibitem{marchioro2012mathematical}
C.~Marchioro and M.~Pulvirenti.
\newblock {\em Mathematical theory of incompressible nonviscous fluids},
  volume~96.
\newblock Springer Science \& Business Media, 2012.

\bibitem{mcduff2012j}
D.~McDuff and D.~Salamon.
\newblock {\em J-holomorphic curves and symplectic topology}, volume~52.
\newblock American Mathematical Soc., 2012.

\bibitem{montgomery1998n}
R.~Montgomery.
\newblock The n-body problem, the braid group, and action-minimizing periodic
  solutions.
\newblock {\em Nonlinearity}, 11(2):363, 1998.

\bibitem{moore1993braids}
C.~Moore.
\newblock Braids in classical dynamics.
\newblock {\em Physical Review Letters}, 70(24):3675, 1993.

\bibitem{o1987stationary}
K.~A. O’Neil.
\newblock Stationary configurations of point vortices.
\newblock {\em Transactions of the American Mathematical Society},
  302(2):383--425, 1987.

\bibitem{palais1979principle}
R.~S. Palais.
\newblock The principle of symmetric criticality.
\newblock {\em Communications in Mathematical Physics}, 69(1):19--30, 1979.

\bibitem{poincare1893theorie}
H.~Poincar{\'e}.
\newblock {\em Th{\'e}orie des tourbillons: Le{\c{c}}ons profess{\'e}es pendant
  le deuxi{\`e}me semestre 1891-92}, volume~11.
\newblock Gauthier-Villars, 1893.

\bibitem{poincare1896solutions}
H.~Poincar{\'e}.
\newblock Sur les solutions p{\'e}riodiques et le principe de moindre action.
\newblock {\em CR Acad. Sci. Paris}, 123:915--918, 1896.

\bibitem{roberts2017morse}
G.~E. Roberts.
\newblock Morse theory and relative equilibria in the planar n-vortex problem.
\newblock {\em Archive for Rational Mechanics and Analysis}, Volume 228, Issue 1, pages 209--236, 2018.

\bibitem{shub1971appendix}
M.~Shub.
\newblock Appendix to smale's paper: Diagonals and relative equilibria.
\newblock In {\em Manifolds—Amsterdam 1970}, pages 199--201. Springer, 1971.

\bibitem{venturelli2001caracterisation}
A.~Venturelli.
\newblock Une caract{\'e}risation variationnelle des solutions de lagrange du
  probleme plan des trois corps.
\newblock {\em Comptes Rendus de l'Acad{\'e}mie des Sciences-Series
  I-Mathematics}, 332(7):641--644, 2001.

\bibitem{helmholtz}
H.~von Helmholtz.
\newblock {{\"U}}ber integrale der hydrodynamischen gleichungen, welche den
  wirbelbewegungen entsprechen.
\newblock {\em Journal f{\"u}r Mathematik Bd. LV. Heft}, 1:4, 1858.

\bibitem{wang2018relative}
Q.~Wang.
\newblock Relative periodic solutions of the n-vortex problem via the
  variational method.
\newblock {\em Archive for Rational Mechanics and Analysis}, pages 1--25,
  2018.

\bibitem{zelati1990periodic}
V.~C. Zelati.
\newblock Periodic solutions for n-body type problems.
\newblock In {\em Annales de l'Institut Henri Poincare (C) Non Linear
  Analysis}, volume 7:5, pages 477--492. Elsevier, 1990.

\bibitem{ziglin1980nonintegrability}
S.~Ziglin.
\newblock Nonintegrability of a problem on the motion of four point vortices.
\newblock In {\em Sov. Math. Dokl}, volume~21, pages 296--299, 1980.

\end{thebibliography}

\end{document}